\documentclass[hidelinks,onefignum,onetabnum]{siamart251216}

\usepackage{amsopn}

\usepackage{float}

\newcommand*{\bbR}{\mathbb{R}}

\newcommand*{\bsa}{\boldsymbol{a}}\newcommand*{\bsb}{\boldsymbol{b}}

\newcommand*{\bst}{\boldsymbol{t}}

\newcommand*{\bsx}{\boldsymbol{x}}\newcommand*{\bsy}{\boldsymbol{y}}

\newcommand*{\rd}{\mathrm{d}}

\newcommand*{\R}{\bbR}

\newcommand{\NN}{\mathbb{N}}
\newcommand{\RR}{\mathbb{R}}
\newcommand{\ZZ}{\mathbb{Z}}

\usepackage{xcolor}
\definecolor{darkred}{RGB}{220,20,60} 
\definecolor{darkblue}{RGB}{0,60,180} 
\definecolor{ikb}{rgb}{0.0, 0.18, 0.65}
\definecolor{darkgreen}{RGB}{0,130,70}
\definecolor{darkorange}{RGB}{180,60,0}
 \usepackage[normalem]{ulem}

\usepackage[normalem]{ulem}
\usepackage{amsfonts}
\usepackage{graphicx} 
\usepackage{epstopdf}
\usepackage{algorithmic}
\usepackage{mathrsfs}
\usepackage{mathtools}
\usepackage{cleveref}

\usepackage{todonotes}
\usepackage{amsmath}
\usepackage{amssymb}
\usepackage{microtype}%
\usepackage{bbm}

\ifpdf
  \DeclareGraphicsExtensions{.eps,.pdf,.png,.jpg}
\else
  \DeclareGraphicsExtensions{.eps}
\fi
\usepackage[backend=biber,style=numeric,maxbibnames=99,giveninits=true,isbn=false,url=false,sortcites=true]{biblatex}
\defbibheading{siamrefheading}{%
  \par\addvspace{.25in}%
  \begin{center}
    \footnotesize\MakeUppercase{References}
  \end{center}
  \addvspace{.15in}\nopagebreak
}

\DeclareFieldFormat*{title}{\mkbibemph{#1}}
\DeclareFieldFormat*{citetitle}{\mkbibemph{#1}}
\DeclareFieldFormat{journaltitle}{#1}

\AddToHook{package/biblatex/after}{%
	\renewbibmacro{in:}{}%
	\AtEveryBibitem{
		\clearfield{day}
		\clearfield{month}
		\clearfield{number}
		\clearlist{location}}
}

\newbibmacro*{pubinstorg+location+date}[1]{%
  \printlist{#1}%
  \newunit
  \printlist{location}%
  \newunit
  \usebibmacro{date}%
  \newunit}

\renewbibmacro*{publisher+location+date}{\usebibmacro{pubinstorg+location+date}{publisher}}
\renewbibmacro*{institution+location+date}{\usebibmacro{pubinstorg+location+date}{institution}}
\renewbibmacro*{organization+location+date}{\usebibmacro{pubinstorg+location+date}{organization}}

\newcommand{\rev}{}

\newsiamremark{remark}{Remark}
\newsiamremark{hypothesis}{Hypothesis}
\crefname{hypothesis}{Hypothesis}{Hypotheses}
\newsiamthm{claim}{Claim}
\newsiamremark{fact}{Fact}
\crefname{fact}{Fact}{Facts}

\allowdisplaybreaks[2]

\headers{Quasi-Monte Carlo and 
 Gauss--Hermite sparse-grid}{Y. Kazashi, Y. Suzuki, and T. Goda}

\title{Optimality of quasi-Monte Carlo methods and \\suboptimality of the sparse-grid Gauss--Hermite rule \\in Gaussian Sobolev spaces\thanks{Submitted to the editors DATE.
\funding{The work of Y.~S. is supported by the Research Council of Finland (decisions 348503, 359181). The work of T.~G.\ is supported by JSPS KAKENHI Grant Number 23K03210.}}}

\author{Yoshihito Kazashi\thanks{Department of Mathematics, University of Manchester, Oxford Road, Manchester M13 9PL, UK 
  (\email{y.kazashi@manchester.ac.uk}).}
\and Yuya Suzuki\thanks{Department of Mathematics and Systems Analysis, School of Science, Aalto University, Espoo, FI-00076 Aalto, Finland
  (\email{yuya.suzuki@aalto.fi}).}
\and Takashi Goda\thanks{Graduate School of Engineering, The University of Tokyo, 7-3-1 Hongo, Bunkyo-ku, Tokyo 113-8656, Japan
  (\email{goda@frcer.t.u-tokyo.ac.jp}).}}

\ifpdf
\hypersetup{
  pdftitle={Optimality of quasi-Monte Carlo methods and suboptimality of the sparse-grid Gauss--Hermite rule in Gaussian Sobolev spaces},
  pdfauthor={Y. Kazashi}
}
\fi

\begin{document}

\maketitle

\begin{abstract}
Optimality of several quasi-Monte Carlo methods and 
suboptimality of the sparse-grid quadrature based on the univariate Gauss--Hermite rule is proved in the Sobolev spaces of mixed dominating smoothness of order $\alpha$, where the optimality is in the sense of worst-case convergence rate. 
For sparse-grid Gauss--Hermite quadrature, lower and upper bounds are established, with rates coinciding up to a logarithmic factor. 
The dominant rate is found to be only $N^{-\alpha/2}$ with $N$ function evaluations, although the optimal rate is known to be $N^{-\alpha}(\ln N)^{(d-1)/2}$. 
The lower bound is obtained by exploiting the structure of the Gauss--Hermite nodes and is independent of the quadrature weights; consequently, no modification of the weights can improve the rate $N^{-\alpha/2}$. 
In contrast, several quasi-Monte Carlo methods with a change of variables are shown to achieve the optimal rate, some up to, and one including, the logarithmic factor.
\end{abstract}



\section{Introduction}

This paper is concerned with numerical integration of multivariate
functions over the whole space in high dimensions. Our focus is on
 integration with respect to the standard Gaussian measure:
\[
I(f):=\int_{\mathbb{R}^{d}}f(\boldsymbol{x})\frac{\mathrm{e}^{-|\boldsymbol{x}|_{\rev{2}}^{2}/2}}{(2\pi)^{d/2}}\mathrm{d}\boldsymbol{x}\approx Q_{N}(f),
\]
where $|\boldsymbol{x}|_\rev{2}=\sqrt{x_{1}^{2}+\dotsb+x_{d}^{2}}$ for $\boldsymbol{x}\in\mathbb{R}^{d}$
is the Euclidean norm, and $Q_{N}$ is a suitable numerical integration
rule using $N$ evaluations of $f$. 
Such problems arise across a
broad range of applications, including data assimilation \cite{Ito.K_Xiong_2000_GaussianFiltersNonlinear},
robotics \cite{Barfoot.T.D_2024_StateEstimationRobotics}, life insurance
\cite{Gerstner.T_eatl_2008_NumericalSimulationAssetLiability}, 
physics \cite{Lubich.C_2008_book_QuantumToClassical}.
They also serve as simpler problems to which more complex integration tasks can be reduced \cite{Ernst.O.G_etal_2025_LearningIntegrate}.

In high dimensions $d\gg1$, two popular classes of numerical integration
rules $Q_{N}$ are sparse-grid quadrature \cite{Bungartz.H_Griebel_2004_sparse_grid_acta}
and quasi-Monte Carlo (QMC) methods \cite{Dick.J_Kuo_Sloan_2013_ActaNumerica}.
Sparse-grid quadrature is a numerical integration technique based
on univariate quadrature rules. 
Unlike fully tensorized quadrature,
it significantly reduces the number of evaluation points, making it
attractive for high-dimensional problems. A commonly used univariate
quadrature for approximating $I(f)$ is the Gauss--Hermite rule \cite{Jia.B_etal_2011_SparseGaussHermiteQuadrature,Radhakrishnan_etal_2016_MultipleSparsegridGauss,Emzir_etal_2023_MultidimensionalProjectionFilters,Bergold.P_Lasser_2024_GaussianWavePacket,YangLi_2025_SparseGridInterpolation}, 
which is the focus of this paper. 
Quasi-Monte Carlo, on the other
hand, is an equal-weight numerical integration. Its core idea is to
carefully choose the quadrature points to improve accuracy. Examples
of point sets include low-discrepancy point sets and sequences, lattice
rules, and digital nets; these classes are not mutually exclusive.
These methods are also widely used for integration with respect to
the Gaussian measure and have found applications across a diverse
range of scientific fields, such as uncertainty quantification \cite{Crevillen_Power_2017,Kubo.K_etal_2022_QuasiMonteCarloSampling}, computational
physics \cite{Jansen.K_etal_2014_QuasiMonteCarloMethods}, finance \cite{Glasserman.P_2003_Book_MCinFinancialEngineering,Lecuyer_2009}, and machine learning \cite{Avron.H_Sindhwani_Yang_Mahoney_2016,MishraEtAl.S_2021_EnhancingAccuracyDeep}.

In this paper, we show that sparse-grid Gauss--Hermite quadrature
is suboptimal, whereas several quasi-Monte Carlo methods are optimal
in terms of convergence rate. More precisely, for $L^{2}$-Sobolev
spaces of degree
$\alpha$, the sparse-grid Gauss--Hermite
rule achieves the convergence rate $O(N^{-\alpha/2})$; several quasi-Monte
Carlo (QMC) methods together with change of variables achieve $O(N^{-\alpha})$
up to a logarithmic factor. 
The rate $O(N^{-\alpha})$ is in fact
optimal, as there exists a matching lower bound established by Dick
et al.~\cite{Dick.J_Irrgeher_Leobacher_Pillichshammer_2018_SINUM_Hermite}.
In this sense, QMC methods considered herein are rate-optimal. In
contrast, we also show that the rate $O(N^{-\alpha/2})$ for the sparse-grid
Gauss--Hermite rule likewise unimprovable, by proving a corresponding
lower bound.

Such a gap in convergence rates between the sparse-grid Gauss--Hermite
rule and QMC methods has also been observed numerically by Dick et
al.~\cite{Dick.J_Irrgeher_Leobacher_Pillichshammer_2018_SINUM_Hermite}
and in subsequent work by Nuyens and one of the present authors \cite{Nuyens.D_Suzuki_2022_ScaledLatticeRd}.
Our theoretical results provide a rigorous explanation of these empirical
findings. 

Analysis of the sparse-grid Gauss--Hermite rule and QMC methods on the whole space has been carried out before, with many papers set in the context of partial differential equations with random coefficients
\cite{
	Chen.P_2018_SparseQuadratureHighdimensionala,
	Dung.D_2021_sparsegrids_lognormal,
	Dung.D_etal_2023_AnalyticitySparsityUncertainty,
	ErnstEtAl.OG_2018_ConvergenceSparseCollocationa,
	Graham.I_etal_2015_Numerische,
	HarbrechtEtAl.H_2016_QuasiMonteCarloMethod,
	Herrmann.L_Schwab_2019_local_lognormal,
	Kazashi.Y_2019_product,
	Nichols.J_Kuo_2014_POD,
	Robbe.P_etal_2017_MultiIndexQuasiMonteCarlo}. 
This paper provides theoretical results for a standard function space, 
without specific structure of the underlying model problem such as the form
of the random field or partial differential equations whose differential operator is given by such a field. In the same vein, aforementioned work by Dick
et al.~\cite{Dick.J_Irrgeher_Leobacher_Pillichshammer_2018_SINUM_Hermite} studied quasi-Monte Carlo methods in a Sobolev space, the space we also consider herein. 
They showed that the QMC method considered there attains the optimal convergence rate by proving matching upper and lower bounds, although the logarithmic factor was not optimal. 
Nuyens and one of the present
authors considered a rank-$1$ lattice QMC over $\mathbb{R}^{d}$
\cite{Nuyens.D_Suzuki_2022_ScaledLatticeRd} for functions with suitable decay.  
We revisit their results and show that the resulting 
QMC method achieves the optimal rate in the Sobolev space considered
herein, although the logarithmic factor is again not optimal. 
In \cite{Suzuki.Y_etal_2025_MobiustransformedTrapezoidalRule}, Hyvönen, Karvonen, and one of the present authors considered a change of variables that maps functions on the real line to the unit interval. 
Although such an approach typically incurs undesirable unboundedness near the boundary, they showed that, with a carefully chosen M\"obius transformation, the resulting mapping sends univariate functions from the Gaussian Sobolev space on the real line to the periodic Sobolev space of the same smoothness on the unit interval. 
In this paper, by applying such a transformation component-wise to $d$-variate functions, we show that the resulting function belongs to two classes of Sobolev spaces that are important in QMC: Korobov spaces and unanchored Sobolev spaces. 
Using this, we establish that QMC rules with the Möbius transformation inherit the convergence rates of the respective spaces. 
These rates turn out to be optimal in Gaussian Sobolev spaces: rank-$1$ lattice rules achieve the optimal rate up to a logarithmic factor, whereas higher-order digital nets achieve the optimal rate including the logarithmic factor. 

Dũng and Nguyen \cite{Dung.D_Nguyen_2023_OptimalNumericalIntegration}
proposed two numerical integration methods that are optimal -- including
the logarithmic factor -- in $L^{p}$-Sobolev spaces for $1<p<\infty$.
One method requires implementing a partition of unity with suitable
properties. The other method appears to be implementable, but the
quadrature weights depend on various indices. The results in this
paper show that equal-weight quadrature rules suffice to obtain an
optimal numerical integration rule.

The remainder of the paper is organized as follows. In \cref{sec:Preliminary} we introduce necessary concepts and state preliminary results. \Cref{sec:GH-SG,sec:QMC} present our main results on the sparse-grid Gauss--Hermite quadrature and QMC methods, respectively. These results build on the authors' previous works \cite{Goda.T_Suzuki_Yoshiki_2018_OptimalOrderQuadrature,Kazashi.Y_Suzuki_Goda_2023_SuboptimalityGaussHermite,Suzuki.Y_etal_2025_MobiustransformedTrapezoidalRule}.
Finally, \cref{sec:conclusion} concludes the paper.

\section{Preliminary\label{sec:Preliminary}}

\subsection{The Sobolev space}

We consider the integration problem in the Gaussian Sobolev space,
also known as the Hermite space of finite smoothness. This space is
defined by 
\[
\|f\|_{H_{\rho}^{\alpha}(\mathbb{R}^{d})}:=\|f\|_{H_{\rho}^{\alpha}}:=\left(\sum_{|\mathbf{r}|_{\infty}\leq\alpha}\|D^{\mathbf{r}}f\|_{L_{\rho}^{2}(\mathbb{R}^{d})}^{2}\right)^{\frac{1}{2}},
\]
where $\rho(\bsx)=\frac{\mathrm{e}^{-|\boldsymbol{x}|_{\rev{2}}^{2}/2}}{(2\pi)^{d/2}}$
and $D^{\mathbf{r}}f$ is the weak derivative of $f$. Here, $\|\cdot\|_{L_{\rho}^{2}(\mathbb{R}^{d})}$
is the norm induced by the weighted $L^{2}$-inner product $\langle f,g\rangle_{L_{\rho}^{2}}:=\int_{\RR^{d}}f(\bsx)g(\bsx)\rho(\bsx)\rd\bsx$
for $f,g\in L_{\rho}^{2}(\R^{d})$.
\rev{Throughout this paper, we denote by  $L_{\mathrm{loc}}^{1}(\mathbb{R}^{d})$ the space of locally integrable functions on  $\mathbb{R}^d$.}

It is known that this space and the so-called Hermite space\rev{ introduced in \cite{IrrgeherEtAl.C_2015_HighdimensionalIntegration$mathbbR^d$} and generalized in \cite{Dick.J_Irrgeher_Leobacher_Pillichshammer_2018_SINUM_Hermite}} coincide
as vector spaces, with equivalent norms. More specifically, Dick et
al.~\cite{Dick.J_Irrgeher_Leobacher_Pillichshammer_2018_SINUM_Hermite}
studied quasi-Monte Carlo methods in the Hermite space and showed
that every element of this space also belongs to $H_{\rho}^{\alpha}(\mathbb{R}^{d})$,
with the embedding being continuous; the reverse inclusion is discussed
in \cite{Kazashi.Y_Suzuki_Goda_2023_SuboptimalityGaussHermite,Dung.D_Nguyen_2023_OptimalNumericalIntegration}.
This space is also known to coincide with the so-called modulation
space, as discussed in \cite{Ehler.M_Groechenig_2023_AbstractApproach}. 
Further properties of Hermite spaces are discussed in~\cite{Gnewuch.M_etal_2022_CountableTensorProducts,Leobacher.G_etal_2023_TractabilityL2approximationIntegration,GnewuchEtAl.M_2024_InfinitedimensionalIntegrationL2L2approximation}.

In the space $H_{\rho}^{\alpha}(\mathbb{R}^{d})$, a lower bound on
the worst-case integration error is known for general algorithms.

\begin{theorem}[\cite{Dick.J_Irrgeher_Leobacher_Pillichshammer_2018_SINUM_Hermite}]\label{thm:lower-all-algorithm}
For $N\geq2$, let $A_{N}\colon H_{\rho}^{\alpha}(\mathbb{R}^{d})\to\mathbb{R}$
be a mapping (linear or nonlinear) that uses only $N$ values as information
about the argument, i.e., it is of the form $A_{N}(f)=\mathcal{I}_{N}(f(\boldsymbol{x}_{1}),\dots,f(\boldsymbol{x}_{N}))$
for a mapping $\mathcal{I}_{N}\colon\mathbb{R}^{N}\to\mathbb{R}$.
Then we have 
\[
	\sup_{ 
		\substack{
		f\in H_{\rho}^{\alpha}(\mathbb{R}^{d})\\
			\|f\|_{H_{\rho}^{\alpha}(\mathbb{R}^{d})}\leq 1
		}	
	}|I(f)-A_{N}(f)|
\geq c_{\alpha,d}\frac{(\ln N)^{\frac{d-1}{2}}}{N^{\alpha}}.
\]
\end{theorem} Hence, any algorithm achieving this convergence rate
$\frac{(\ln N)^{\frac{d-1}{2}}}{N^{\alpha}}$ may be regarded as
\emph{optimal}, in terms of the worst-case error convergence rate.
We will show that several QMC methods attain a rate of the form $\frac{(\ln N)^{\gamma}}{N^{\alpha}}$,
where the exponent $\gamma$ depends on the method. The \rev{exponent} $\gamma=\frac{d-1}{2}$
turns out to be attainable, hence optimal in the sense above. In contrast,
the sparse-grid quadrature based on the Gauss--Hermite rule is of
the rate $N^{-\alpha/2}$, up to a logarithmic factor, which is subotimal
in the sense of worst-case error convergence rate.

Note that \cite[Theorem 1]{Dick.J_Irrgeher_Leobacher_Pillichshammer_2018_SINUM_Hermite}
is formulated for linear algorithms. As noted in \cite[Section 3]{Dick.J_Irrgeher_Leobacher_Pillichshammer_2018_SINUM_Hermite},
however, in view of for example \cite{Bakhvalov.N.S_1971_OptimalityLinearMethods}
(see also \cite[Section 4.2]{Novak.E_Wozniakowski_book_1}), their
lower bound also holds for nonlinear algorithms. Moreover, the implied
constant in the result above may differ from that in \cite[Theorem 1]{Dick.J_Irrgeher_Leobacher_Pillichshammer_2018_SINUM_Hermite},
since the present bound is formulated for the Sobolev space, whereas
\cite[Theorem 1]{Dick.J_Irrgeher_Leobacher_Pillichshammer_2018_SINUM_Hermite}
considers a norm-equivalent space, namely the Hermite space. \rev{Further characterizations of Hermite spaces are investigated in \cite{Leobacher.G_etal_2023_TractabilityL2approximationIntegration}.}

Let $H_{k}$ for $k\in\mathbb{N}_{0}:=\mathbb{N}\cup\{0\}$ denote
the normalized probabilistic Hermite polynomial of degree $k$, defined
by 
\[
H_{k}(x)=\frac{(-1)^{k}}{\sqrt{k!}}\mathrm{e}^{x^{2}/2}\frac{\rd^{k}}{\rd x^{k}}\mathrm{e}^{-x^{2}/2}.
\]
The multivariate counterpart is defined by $H_{\mathbf{k}}(\bsx)=\prod_{j=1}^{d}H_{k_{j}}(x_{j})$
for $\mathbf{k}\in\mathbb{N}_{0}^{d}$. The collection $(H_{\mathbf{k}})_{\mathbf{k}\in\mathbb{N}_{0}^{d}}$
forms a complete orthonormal system for $L_{\rho}^{2}(\R^{d})$, with
the normalization above ensuring $\|H_{\mathbf{k}}\|_{L_{\rho}^{2}}=1$
for all $\mathbf{k}\in\mathbb{N}_{0}^{d}$.

Throughout the paper, for mappings $f,g\colon S\to[0,\infty]$ on
any set $S$, we write $f\lesssim g$ if there exists a constant $c\in(0,\infty)$
such that $f(x)\leq cg(x)$ holds for all $x\in S$. We write $f\asymp g$
if both $f\lesssim g$ and $g\lesssim f$. If the constant may depend
on specific parameters, we indicate this by subscripts. For example,
$f\lesssim_{d}g$ means there exists a constant $c_{d}\in(0,\infty)$,
which may depend on $d$, such that $f(x)\leq c_{d}\,g(x)$ for all
$x\in S$; likewise, $f\asymp_{d}g$ means $f\lesssim_{d}g$ and $g\lesssim_{d}f$. 
\rev{We also use the following notations: $\mathbf{a}!:=\prod_{j=1}^{d}a_j!$ and $\binom{\mathbf{a}}{\mathbf{b}}:=\prod_{j=1}^{d}\binom{a_j}{b_j}$ for $\mathbf{a},\mathbf{b}\in \mathbb{N}^d_0$.} 
\rev{Furthermore, for $\mathbf{k}\in\mathbb{N}_{0}^{d}$ we let $|\mathbf{k}|:=k_1+\dotsb+k_d$.}
\subsection{Tensorised operators\label{subsec:tensorised-operators}}

In this section we define the tensorised operator $T_{1}\otimes\dotsb\otimes T_{d}\colon H_{\rho}^{\alpha}(\mathbb{R}^{d})\to H_{\rho}^{\alpha}(\mathbb{R}^{d})$
based on bounded linear operators $T_{j}\colon H_{\rho}^{\alpha}(\mathbb{R})\to H_{\rho}^{\alpha}(\mathbb{R})$,
$j=1,\dots,d$. The main motivation for this subsection is to construct
such an expression on $H_{\rho}^{\alpha}(\mathbb{R}^{d})$ without
using the norm-equivalence with the Hermite space, which will be used
to discuss the Smolyak algorithm. 
\begin{proposition} \label{prop:dense}The
$\mathbb{R}$-vector space spanned by $d$-variate polynomials $\mathcal{P}:=\mathrm{span}\Big\{\prod_{j=1}^{d}H_{k_{j}}\mid k_{j}\in\mathbb{N}_{0},\,j=1,\dots,d\Big\}$
is dense in $H_{\rho}^{\alpha}(\mathbb{R}^{d})$. 
\end{proposition} 
\begin{proof}
Fix $f\in H_{\rho}^{\alpha}$. Following the argument in \cite[Lemma 2.1]{Kazashi.Y_Suzuki_Goda_2023_SuboptimalityGaussHermite}
we see 
\[
\langle D^{\mathbf{r}}f,H_{\mathbf{k}}\rangle_{L_{\rho}^{2}}=\biggl(\prod_{j=1}^{d}\prod_{m=1}^{\mathrm{r}_{j}}\bigl(k_{j}+m\bigr)\biggr)^{1/2}\langle f,H_{\mathbf{k}+\mathbf{r}}\rangle_{L_{\rho}^{2}}\quad\text{for }\mathbf{k}\in\mathbb{N}_{0}^{d},\ \mathbf{r}\in\{0,\dots,\alpha\}^{d}.
\]
With this, using the Parseval identity for derivatives we see that
the polynomial $f_{N}=\sum_{|\mathbf{k}|\leq N}\langle f,H_{\mathbf{k}}\rangle_{L_{\rho}^{2}}H_{\mathbf{k}}$
can be made arbitrarily close to $f$ in $H_{\rho}^{\alpha}$: 
\[
\|f-f_{N}\|_{H_{\rho}^{\alpha}}^{2}=\sum_{|\mathbf{r}|_{\infty}\leq\alpha}\sum_{|\mathbf{k}|>N}\biggl(\prod_{j=1}^{d}\prod_{m=1}^{\mathrm{r}_{j}}\bigl(k_{j}+m\bigr)\biggr)|\langle f,H_{\mathbf{k}+\mathbf{r}}\rangle_{L_{\rho}^{2}}|^{2}.
\]
\end{proof}

Let $T_{j}\colon H_{\rho}^{\alpha}(\mathbb{R})\to H_{\rho}^{\alpha}(\mathbb{R})$
be bounded linear operators. Define $T_{1}\otimes\dotsb\otimes T_{d}:\mathcal{P}\to H_{\rho}^{\alpha}(\mathbb{R}^{d})$
by defining 
\begin{equation*}
\left[T_{1}\otimes\dotsb\otimes T_{d}\Bigg(\prod_{j=1}^{d}H_{k_{j}}\Bigg)\right](\boldsymbol{x}) =\prod_{j=1}^{d}[T_{j}(H_{k_{\rev{j}}})](x_{j})\qquad\text{for }\prod_{j=1}^{d}H_{k_{j}}\in\mathcal{P}
\end{equation*}
and extending it linearly to $\mathcal{P}$. 
Then, using the argument for tensor product Hilbert spaces \cite[Section VIII.10]{Reed.M_Simon_1980_book}
we see $\sup_{p\in\mathcal{P},\|p\|_{H_{\rho}^{\alpha}(\mathbb{R}^{d})}=1}\|T_{1}\otimes\dotsb\otimes T_{d}(p)\|_{\rev{H_{\rho}^{\alpha}(\mathbb{R}^{d})}}=\prod_{j=1}^{d}\|T_{j}\|_{H_{\rho}^{\alpha}(\mathbb{R})\to H_{\rho}^{\alpha}(\mathbb{R})}$
on $\mathcal{P}$, and from \cref{prop:dense}, the operator
$T_{1}\otimes\dotsb\otimes T_{d}$ extends continuously to $H_{\rho}^{\alpha}(\mathbb{R}^{d})$
with the same operator norm.

\rev{Note that this construction applies also when the $T_{j}$ are
bounded linear functionals. In such a case, since $T_{j}(f)$ is a
constant function for $f\in H_{\rho}^{\alpha}(\mathbb{R})$ and the
total measure is normalised to $1$, we have $\|T_{j}(f)\|_{H_{\rho}^{\alpha}(\mathbb{R})}=|T_{j}(f)|$.
From this, we see that $T_{j}$ is also a bounded operator from $H_{\rho}^{\alpha}(\mathbb{R})$
to itself with $\|T_{j}\|_{H_{\rho}^{\alpha}(\mathbb{R})\to H_{\rho}^{\alpha}(\mathbb{R})}=\|T_{j}\|_{H_{\rho}^{\alpha}(\mathbb{R})\to\mathbb{R}}$.
Similarly, for the functional $T_{1}\otimes\dotsb\otimes T_{d}$
on $\mathcal{P}$ constructed as above, $T_{1}\otimes\dotsb\otimes T_{d}(p)$
is a constant function so we have $|T_{1}\otimes\dotsb\otimes T_{d}(p)|=\|T_{1}\otimes\dotsb\otimes T_{d}(p)\|_{H_{\rho}^{\alpha}(\mathbb{R}^{d})}$.
Hence, $T_{1}\otimes\dotsb\otimes T_{d}$ satisfies $\sup_{p\in\mathcal{P},\|p\|_{H_{\rho}^{\alpha}(\mathbb{R}^{d})}=1}|T_{1}\otimes\dotsb\otimes T_{d}(p)|=\prod_{j=1}^{d}\|T_{j}\|_{H_{\rho}^{\alpha}(\mathbb{R})\to\mathbb{R}}$ 
 and admits a continuous extension to $H_{\rho}^{\alpha}(\mathbb{R}^{d})$ with the same functional norm.
}%
Operators with the symbol $\otimes$ in this paper are constructed
explicitly in this manner.

\subsection{Sobolev embedding for spaces with dominating mixed smoothness}

Let $W_{\mathrm{mix}}^{1,2}(\mathbb{R}^{d})$ denote the Sobolev space
of dominating mixed first-order smoothness, equipped with the norm
\[
\|v\|_{W_{\mathrm{mix}}^{1,2}(\mathbb{R}^{d})}:=\left(\sum_{|\mathbf{r}|_{\infty}\leq1}\|D^{\mathbf{r}}v\|_{L^{2}(\mathbb{R}^{d})}^{2}\right)^{\frac{1}{2}},
\]
where $\|\cdot\|_{L^{2}(\mathbb{R}^{d})}$ is the standard $L^{2}(\mathbb{R}^{d})$-norm. 
Later, we will use the following Sobolev inequality for $W_{\mathrm{mix}}^{1,2}(\mathbb{R}^{d})$,
which is also of independent interest.

\begin{proposition}\label{prop:1st-order-embedding}Every element of
$W_{\mathrm{mix}}^{1,2}(\mathbb{R}^{d})$ admits a bounded continuous
representative $v$ satisfying 
\begin{align}
\sup_{\boldsymbol{x}\in\mathbb{R}^{d}}|v(\boldsymbol{x})| &
 \leq\|v\|_{W_{\mathrm{mix}}^{1,2}(\mathbb{R}^{d})},
\end{align}
and
\begin{align}
|v(\boldsymbol{x})-v(\boldsymbol{y})| & \leq d^{3/4}\|v\|_{W_{\mathrm{mix}}^{1,2}(\mathbb{R}^{d})}|\boldsymbol{x}-\boldsymbol{y}|_{2}^{1/2}\qquad\text{for all }\boldsymbol{x},\boldsymbol{y}\in\mathbb{R}^{d}.\label{eq:Hoelder}
\end{align}
\end{proposition} 
\begin{proof}
We will first prove the statement for infinitely differentiable functions
with compact support, $C_{\mathrm{c}}^{\infty}(\mathbb{R}^d)$, and
then extend it by a density argument. For $v\in C_{\mathrm{c}}^{\infty}(\mathbb{R}^{d})$,
following \cite[Proof of Theorem 8.8]{Brezis.H_book_2010} we get
\begin{align*}
|v(x_{1},\dots,x_{d})|^{2} & \leq2\sqrt{\int_{\mathbb{R}}|v(s,x_{2},\dots,x_{d})|^{2}\mathrm{d}s}\sqrt{\int_{\mathbb{R}}|D^{(1,0,\dots,0)}v(s,x_{2},\dots,x_{d})|^{2}\mathrm{d}s}\\
 & \leq\int_{\mathbb{R}}|v(s,x_{2},\dots,x_{d})|^{2}\mathrm{d}s+\int_{\mathbb{R}}|D^{(1,0,\dots,0)}v(s,x_{2},\dots,x_{d})|^{2}\mathrm{d}s.
\end{align*}
Applying the same argument for $v(s,x_{2},\dots,x_{d})$ and $D^{(1,0,\dots,0)}v(s,x_{2},\dots,x_{d})$
with respect to the second argument $x_{2}$, and repeating the procedure
for $x_{3}$ up to $x_{d}$, we obtain 
\[
|v(x_{1},\dots,x_{d})|^{2}\leq\sum_{r_{1},\dots,r_{d}\in\{0,1\}}\int_{\mathbb{R}}\dotsb\int_{\mathbb{R}}|D^{(r_{1},\dots,r_{d})}v(s_{1},\dots,s_{d})|^{2}\mathrm{d}s_{1}\dotsb\mathrm{d}s_{d},
\]
which gives $\sup_{\boldsymbol{x}\in\mathbb{R}^{d}}|v(\boldsymbol{x})|\leq\|v\|_{W_{\mathrm{mix}}^{1,2}(\mathbb{R}^{d})}$
for $v\in C_{\mathrm{c}}^{\infty}(\mathbb{R}^{d})$.

Now fix any representative $v\in W_{\mathrm{mix}}^{1,2}(\mathbb{R}^{d})$.
The density of $C_{\mathrm{c}}^{\infty}(\mathbb{R}^{d})$ in $W_{\mathrm{mix}}^{1,2}(\mathbb{R}^{d})$
can be checked using a standard argument from the Sobolev space theory,
for example by approximation via smoothly truncated mollified functions.
Thus, there exists a sequence $(v_{n})\subset C_{\mathrm{c}}^{\infty}(\mathbb{R}^{d})$
that approximates $v$ in $W_{\mathrm{mix}}^{1,2}(\mathbb{R}^{d})$.
The sequence $(v_{n})$ also belongs to the Banach space $C_{\mathrm{b}}(\mathbb{R}^{d})$
of bounded continuous functions equipped with the supremum norm. The
inequality already proved gives 
\begin{align*}
\sup_{\boldsymbol{y}\in\mathbb{R}^{d}}|v_{n}(\boldsymbol{y})-v_{m}(\boldsymbol{y})| & \leq\|v_{n}-v_{m}\|_{W_{\mathrm{mix}}^{1,2}(\mathbb{R}^{d})},
\end{align*}
so, $(v_{n})$ is Cauchy in $C_{\mathrm{b}}(\mathbb{R}^{d})$ and
thus has a limit $\tilde{v}$ in $C_{\mathrm{b}}(\mathbb{R}^{d})$,
which in particular converges to $\tilde{v}$ pointwise. Since a subsequence
of $(v_{n})$ converges to $v$ a.e., we conclude $\tilde{v}=v$ a.e.
This also implies that both sides of $\sup_{\boldsymbol{x}\in\mathbb{R}^{d}}|v_{n}(\boldsymbol{x})|\leq\|v_{n}\|_{W_{\mathrm{mix}}^{1,2}(\mathbb{R}^{d})}$
admit limits, which yields the desired inequality for $v\in W_{\mathrm{mix}}^{1,2}(\mathbb{R}^{d})$.

		To show the second inequality, again take $v\in C_{\mathrm{c}}^{\infty}(\mathbb{R}^{d})$.
		For any $\boldsymbol{x},\boldsymbol{y}\in\mathbb{R}^{d}$, we have
		\begin{align*}
		|v(\boldsymbol{x})-v(\boldsymbol{y})| & \leq|v(\boldsymbol{x})-v(y_{1},x_{2},\dots,x_{d})|\\
		& \quad+\sum_{j=1}^{d-2}|v(y_{1},\dots,y_{j},x_{j+1},\dots,x_{d})-v(y_{1},\dots,y_{j+1},x_{j+2},\dots,x_{d})|\\
		& \quad+|v(y_{1},\dots,y_{d-1},x_{d})-v(\boldsymbol{y})|.
		\end{align*}
		The first term is controlled by 
		\[
		|v(\boldsymbol{x})-v(y_{1},x_{2},\dots,x_{d})|\leq\left(\int_{\mathbb{R}}|D^{(1,0,\dots,0)}v(s,x_{2},\dots,x_{d})|^{2}\mathrm{d}s\right)^{1/2}|x_{1}-y_{1}|^{1/2},
		\]
		which can be further bounded by $\|v\|_{W_{\mathrm{mix}}^{1,2}(\mathbb{R}^{d})}|x_{1}-y_{1}|^{1/2}$,
		following the argument above. The other terms are treated similarly.
		Therefore, we obtain $|v(\boldsymbol{x})-v(\boldsymbol{y})|\leq\|v\|_{W_{\mathrm{mix}}^{1,2}(\mathbb{R}^{d})}\sum_{j=1}^{d}|x_{j}-y_{j}|^{1/2}$.
		Finally, using the Hölder inequality with $p=4,q=4/3$ we get 
		\[
		\sum_{j=1}^{d}|x_{j}-y_{j}|^{1/2}\leq\bigg(\sum_{j=1}^{d}|x_{j}-y_{j}|^{2}\bigg)^{1/4}\big(\sum_{j=1}^{d}\;1\big)^{3/4}=d^{3/4}|\boldsymbol{x}-\boldsymbol{y}|_{2}^{1/2},
		\]
		which shows the second inequality for $C_{\mathrm{c}}^{\infty}(\mathbb{R}^{d})$.
		Using the density of $C_{\mathrm{c}}^{\infty}(\mathbb{R}^{d})$ completes
		the proof. 
\end{proof}

\section{Gauss--Hermite sparse-grid quadrature}\label{sec:GH-SG}

A sparse-grid quadrature is constructed based on univariate quadrature
rules. For $\ell\in\mathbb{N}_{0}$, let $Q_{\ell}^{\mathrm{uni}}\colon H_{\rho}^{\alpha}(\mathbb{R})\to\mathbb{R}$
a quadrature rule with $n_{\ell}$ quadrature points, with $Q_{0}^{\mathrm{uni}}:=0$,
i.e., $Q_{0}^{\mathrm{uni}}(f)=0$ for all $f\in H_{\rho}^{\alpha}(\mathbb{R})$. 
Define $\Delta_{0}:=Q_{0}^{\mathrm{uni}}=0$, and for $\ell\ge1$,
$\Delta_{\ell}:=Q_{\ell}^{\mathrm{uni}}-Q_{\ell-1}^{\mathrm{uni}}$.
Since constant functions belong to $H_{\rho}^{\alpha}(\mathbb{R})$
and $H_{\rho}^{\alpha}(\mathbb{R}^{d})$, the operator $\Delta_{\ell_{1}}\otimes\cdots\otimes\Delta_{\ell_{d}}\colon H_{\rho}^{\alpha}(\mathbb{R}^{d})\to\mathbb{R}$
is well-defined; see \cref{subsec:tensorised-operators}.
Given an index set $\Lambda\subset\mathbb{N}^{d}$, the sparse-grid
quadrature $S_{\Lambda}\colon H_{\rho}^{\alpha}(\mathbb{R}^{d})\to\mathbb{R}$
is defined by 
\begin{equation}
S_{\Lambda}:=\sum_{(\ell_{1},\dots,\ell_{d})\in\Lambda}\Delta_{\ell_{1}}\otimes\dotsb\otimes\Delta_{\ell_{d}}.\label{eq:Smolyak}
\end{equation}

The purpose of this section is to analyse the error of $S_{\Lambda}$
when $Q_{\ell}^{\mathrm{uni}}$ is the Gauss--Hermite quadrature.
In what follows, unless otherwise stated, we use $S_{\Lambda}$ to
denote the operator in \eqref{eq:Smolyak} with Gauss--Hermite quadrature.

\subsection{Lower bounds}

In this section, we present lower bounds for two types of Smolyak
quadratures: one for a general class of index sets (\cref{thm:GH-lower-bound-general})
and the other for the classical isotropic case (\cref{thm:lower-isotropic-Smolyak}).
In both cases, the dominant decay rate turns out to be half of the
optimal rate as in \cref{thm:lower-all-algorithm}.

For the isotropic case, we also provide a matching upper bound in
\cref{sec:upper-bound}, implying that our lower bound is
sharp up to a logarithmic factor. Furthermore, a lower bound with
the same polynomial factor holds for general index sets, implying
that using an anisotropic index set, e.g., to exploit anisotropy in
the target functions, does not improve the convergence rate.

To make the statement of the theorem concise, we say that an index
set $\Lambda\subset\mathbb{N}^{d}$ is \textit{downward-closed in
one direction} if there is a coordinate index $j\in\{1,\dots,d\}$
such that whenever $(1,\dots1,\ell,1\dots,1)\in\Lambda$ holds, where
$\ell$ is in the $j$-th position, then for all integers $k\leq\ell$
we also have $(1,\dots1,k,1\dots,1)\in\Lambda$. Clearly, downward-closed
set is downward-closed in one direction. \begin{theorem} \label{thm:GH-lower-bound-general}
Suppose that a finite index set $\Lambda\subset\mathbb{N}^{d}$ is
downward-closed in one direction. Let $S_{\Lambda}$ be the corresponding
sparse-grid quadrature based on the Gauss--Hermite rule. Let $N$
be the total number of the point evaluations used in $S_{\Lambda}$.
Then, we have 
\[
\sup_{f\neq0}\frac{|I(f)-S_{\Lambda}(f)|}{\|f\|_{H_{\rho}^{\alpha}}}\geq c_{\alpha}N^{-\alpha/2}.
\]
 where the constant $c_{\alpha}>0$ is independent of $N$ and
$d$. \end{theorem} 
\begin{proof}
Our strategy follows \cite[Section 3.1]{Kazashi.Y_Suzuki_Goda_2023_SuboptimalityGaussHermite}. 
The heart of the matter is to construct a suitable fooling function,
which we denote by $h_{\Lambda}$ below. %
Without loss of generality, we assume that $\Lambda$ is downward-closed
in the first coordinate. Then, for some $\overline{\ell}_{1}\in\mathbb{N}$
we have $(\ell_{1},1\dots,1)\in\Lambda$ for all $\ell_{1}=1,\dots,\overline{\ell}_{1}$.
For example, if $\Lambda$ is the classical isotropic index set $\Lambda=\{\boldsymbol{\ell}\in\mathbb{N}^{d}\mid\,|\rev{\boldsymbol{\ell}|} \leq L\}$
then $\overline{\ell}_{1}=L-d+1$.

Let $n:=n_{\overline{\ell}_{1}}$ denote the number of quadrature
points used in the univariate Gauss--Hermite rule $Q_{\overline{\ell}_{1}}^{\mathrm{uni}}$,
and let $\xi_{1},\dots,\xi_{n}$ be the roots of the univariate Hermite
polynomial $H_{n}$ of degree $n$. Following~\cite{Kazashi.Y_Suzuki_Goda_2023_SuboptimalityGaussHermite},
define the function $p_{n}(t)\colon\mathbb{R}\to\mathbb{R}$ by 
\[
p_{n}(t)=\begin{cases}
\left(\frac{t-\xi_{j}}{\xi_{j+1}-\xi_{j}}\right)^{\alpha}\left(1-\frac{t-\xi_{j}}{\xi_{j+1}-\xi_{j}}\right)^{\alpha} & \begin{array}{l}
\text{ if there exists }j\in\{1,\ldots,n\}\\
\text{ such that }t\in\left[\xi_{j},\xi_{j+1}\right],
\end{array}\\
0 & \text{ otherwise. }
\end{cases}
\]
With this, define

\begin{equation}
h_{\Lambda}(\boldsymbol{x}):=p_{n}(x_{1})\quad\text{for }\boldsymbol{x}=(x_{1},\dots,x_{d})\in\mathbb{R}^{d}.\label{eq:def-fooling}
\end{equation}

First, we will show 
\begin{equation}
\frac{|I(h_{\Lambda})-S_{\Lambda}(h_{\Lambda})|}{\|h_{\Lambda}\|_{H_{\rho}^{\alpha}}}=\frac{\int_{\mathbb{R}}p_{n}(x)\rho(x)\mathrm{d}x
}{\left(\sum_{r=0}^{\alpha}\|D^{r}p_{n}\|_{L_{\rho}^{2}}^{2}\right)^\rev{{1/2}}},\label{eq:1d-RHS}
\end{equation}
so that the problem reduces to a one-dimensional integration problem.
Recall $\Delta_{1}(1)=1$ and $\Delta_{\ell}(1)=0$ for $\ell\neq1$
because the Gauss--Hermite rule integrates constant functions exactly.
Thus, it follows 
\begin{align*}
S_{\Lambda}(h_{\Lambda}) & =\sum_{(\ell_{1},\dots,\ell_{d})\in\Lambda}\Delta_{\ell_{1}}(p_{n})\Delta_{\ell_{2}}(1)\cdots\Delta_{\ell_{d}}(1)\\
 & =\sum_{\substack{(\ell_{1},\dots,\ell_{d})\in\Lambda\\
\ell_{2}=\dotsb=\ell_{d}=1
}
}\Delta_{\ell_{1}}(p_{n})\Delta_{\ell_{2}}(1)\cdots\Delta_{\ell_{d}}(1)=Q_{\overline{\ell}_{1}}^{\mathrm{uni}}(p_{n}),
\end{align*}
Hence, $S_{\Lambda}(h_{\Lambda})=0$ and $|I(h_{\Lambda})-S_{\Lambda}(h_{\Lambda})|
=\int_{\mathbb{R}}p_{n}(x)\rho(x)\mathrm{d}x$.
As for the norm, since $h_{\Lambda}(\boldsymbol{x})=p_{n}(x_{1})$
we obtain 
\begin{align*}
\|h_{\Lambda}\|_{H_{\rho}^{\alpha}} & =\left(\sum_{|\mathbf{r}|_{\infty}\leq\alpha}\|D^{(r_{1},0,\dots,0)}h_{\Lambda}\|_{L_{\rho}^{2}}^{2}\right)^{\frac{1}{2}}=\left(\sum_{r=0}^{\alpha}\|D^{r}p_{n}\|_{L_{\rho}^{2}}^{2}\right)^{\frac{1}{2}}\rev{.}
\end{align*}
Therefore, \eqref{eq:1d-RHS} holds.

The right-hand side of \eqref{eq:1d-RHS} corresponds to the expression
in the last displayed equation in \cite[Proof of Lemma 3.1]{Kazashi.Y_Suzuki_Goda_2023_SuboptimalityGaussHermite},
which is further analysed in \cite[Theorem 3.2]{Kazashi.Y_Suzuki_Goda_2023_SuboptimalityGaussHermite}.
There, it is shown 
\[
\frac{\int_{\mathbb{R}}p_{n}(x)\rho(x)\mathrm{d}x}{\left(\sum_{r=0}^{\alpha}\|D^{r}p_{n}\|_{L_{\rho}^{2}}^{2}\right)^{\rev{1/2}}}
\geq c_{\alpha}n^{-\alpha/2},
\]
for some constant $c_{\alpha}>0$ independent of $N$.
From
\[ n\leq \#\left(\sum_{\substack{(\ell_{1},\dots,\ell_{d})\in\Lambda\\
\ell_{2}=\dotsb=\ell_{d}=1}}\Delta_{\ell_{1}}\otimes\dotsb\otimes\Delta_{\ell_{d}}\right) \leq \#\left(\sum_{(\ell_{1},\dots,\ell_{d})\in\Lambda}\Delta_{\ell_{1}}\otimes\dotsb\otimes\Delta_{\ell_{d}}\right)=N,\]
\rev{where $\#(Q)$ denotes the total number of quadrature nodes used in a quadrature rule $Q$, }the statement follows. 
\end{proof}

For the particular choice of index set 
\[
\Lambda_{L}=\{\boldsymbol{\ell}\in\mathbb{N}^{d}\mid|\rev{\boldsymbol{\ell}|} \leq L\},\quad
\rev{\text{where }\ |\boldsymbol{\ell}|=\ell_1+\dots+\ell_d,}
\]
we establish a lower bound of a slower decay, now including a logarithmic
factor that grows with the dimension; see \cref{thm:lower-isotropic-Smolyak}.

The index set $\Lambda_{L}$ is of particular interest because it
corresponds to the classical (isotropic) sparse-grid quadrature, which
was Smolyak’s original choice of index set. Note that $S_{\Lambda_{L}}=S_{\Lambda_{L}^{0}}$,
where 
\[
\Lambda_{L}^{0}=\{\boldsymbol{\ell}\in\mathbb{N}_{0}^{d}\mid|\rev{\boldsymbol{\ell}|} \leq L\}
\]
because $\Delta_{0}=0$. Moreover, if $L<d$ then $S_{\Lambda_{L}}=0$,
since $\ell_{1}+\dotsb+\ell_{d}<L<d$ implies that at least one of
the indices is zero, say $\ell_{j}=0$, and thus $\Delta_{\ell_{j}}=0$.

\begin{theorem}\label{thm:lower-isotropic-Smolyak}Let $S_{\Lambda_{L}}$
be the sparse-grid quadrature based on the Gauss--Hermite rule
for $\Lambda_{L}=\{\boldsymbol{\ell}\in\mathbb{N}^{d}\mid|\rev{\boldsymbol{\ell}|} \leq L\}$.
Suppose that the Gauss--Hermite rule $Q_{\ell}^{\mathrm{uni}}$ for
univariate functions employs $n_{\ell}$ quadrature points for every
$\ell\geq1$,
where the parameters \(n_\ell\) are chosen so that
\(
    n_\ell\neq n_{\ell'}
\) whenever \(\ell\neq \ell'\),
and, for some \(M_0>0\) and \(M>1\),
\[
M^{\ell-1}\leq n_{\ell}\leq M_{0}(M^{\ell}-1),\qquad  \ell\ge1.
\]
Let $N$ denote the total number of the function evaluations used
in $S_{\Lambda_{L}}$. Then, we have 
\[
\sup_{f\neq0}\frac{|I(f)-S_{\Lambda_{L}}(f)|}{\|f\|_{H_{\rho}^{\alpha}}}\geq c_{\alpha,d}N^{-\alpha/2}(\ln N)^{\frac{a}{2}(d-1)}.
\]
where the constant $c_{\alpha,d}>0$ is independent of $N$. \end{theorem} 
\begin{proof}
Define $h_{\Lambda_{L}}(\boldsymbol{x}):=p_{n_{L-d+1}}(x_{1})$ for
$\boldsymbol{x}=(x_{1},\dots,x_{d})\in\mathbb{R}^{d}$. From the proof
of \cref{thm:GH-lower-bound-general} we know 
\[
\frac{|I(h_{\Lambda_{L}})-S_{\Lambda_{L}}(h_{\Lambda_{L}})|}{\|h_{\Lambda_{L}}\|_{H_{\rho}^{\alpha}}}\geq c_{\alpha}n_{L-d+1}^{-\alpha/2},
\]
so it remains to relate $N$ and $n_{L-d+1}$.

Given the choice $\Lambda_{L}=\bigl\{\boldsymbol{\ell}=(\ell_{1},\dots,\ell_{d})\in\mathbb{N}^{d}:|\rev{\boldsymbol{\ell}|} \le L\bigr\}$,
the total number of point evaluation $N$ can be written explicitly
as 
\begin{equation}
N=\sum_{\boldsymbol{\ell}\in\Lambda_{L}}\prod_{k=1}^{d}m_{\ell_{k}},
\end{equation}
where \(m_\ell\) denotes the number of nodes of the univariate rule
\(Q_\ell^{\mathrm{uni}}\) that are not used in any univariate rule with a lower level index $\ell'<\ell$. 
 More precisely, let \(X_\ell\subset \mathbb{R}\) be the set of quadrature nodes used in 
\(Q_\ell^{\mathrm{uni}}\).  We define
\(
    \Theta_\ell
    :=
    X_\ell\setminus \bigcup_{j=1}^{\ell-1}X_j,
\) for \(\ell\in\mathbb{N}\), 
where the union is understood to be empty for \(\ell=1\), and set
\(
    m_\ell:=\#\Theta_\ell
\). 
Since Hermite polynomials of different degrees can share only the root \(0\)
\cite[Lemma 3.2]{Garcia-Ferrefo.M_Gomez-Ullate_2015_Hermitezeros}, and 
since the quadrature orders \(n_\ell\) are assumed to be distinct for different levels, at most one node of \(Q_\ell^{\mathrm{uni}}\), namely \(0\), can have already been used at a lower level. Hence
\(
    n_\ell-1\le m_\ell\le n_\ell
\). Therefore, the total number of evaluations $N$ satisfies
	\[
	N%
	\asymp_d 
	\sum_{\boldsymbol{\ell}\in\Lambda_{L}}\prod_{k=1}^{d}n_{\ell_{k}}.
	\]

Hence $M^{\ell-1}\leq n_{\ell}\leq M_{0}(M^{\ell}-1)$ implies 
\[
M^{-d}\sum_{\boldsymbol{\ell}\in\Lambda_{L}}M^{|\rev{\boldsymbol{\ell}|} }\lesssim_d N \lesssim_d M_{0}^{d}\sum_{\boldsymbol{\ell}\in\Lambda_{L}}M^{|\rev{\boldsymbol{\ell}|} },
\]
and thus together with $\sum_{\boldsymbol{\ell}\in\Lambda_{L}}M^{|\rev{\boldsymbol{\ell}|} }=\sum_{k=d}^{L}M^{k}\sum_{\substack{\boldsymbol{\ell}\in\mathbb{N}^{d}\\
|\rev{\boldsymbol{\ell}|} =k
}
}1$ and $\#\{\boldsymbol{\ell}\in\mathbb{N}^{d}:|\rev{\boldsymbol{\ell}|} =k\}=\frac{(k-1)!}{(k-d)!(d-1)!}=\Theta_{d}(k^{d-1})$
we obtain 
\begin{equation}
\underline{C}_{d}M^{L}L^{d-1} \leq \underline{C}'_{d}M^{L}\sum_{\substack{\boldsymbol{\ell}\in\mathbb{N}^{d}\\
|\rev{\boldsymbol{\ell}|} =L
}
}1\leq N\le\overline{C}_{d}L^{d-1}\sum_{k=0}^{L}M^{k}\leq\frac{\overline{C}_{d}M}{M-1}L^{d-1}M^{L}.\label{eq:isotropic-N-bound}
\end{equation}
\sloppy{\noindent 
Here, the implied constants depend on $d$.   Therefore, $N\asymp_{d}n_{L}(\ln n_{L})^{d-1}\asymp_{d}n_{L-d+1}(\ln n_{L-d+1})^{d-1}$.
Since \((\ln n_{L-d+1})^{d-1}\lesssim_d  n_{L-d+1}\), this implies $\ln N\lesssim_{d} \ln n_{L-d+1}$.}
Hence, we conclude
\[
n_{L-d+1}^{-\alpha/2}\asymp_{d}N^{-\alpha/2}(\ln n_{L-d+1})^{\frac{a}{2}(d-1)}\gtrsim_{d}N^{-\alpha/2}(\ln N)^{\frac{a}{2}(d-1)},
\]
which completes the proof. 
\end{proof}
These bounds turn out to be sharp in the sense that for the standard
isotropic Smolyak's algorithm we have a matching upper bound, as we
will see in \cref{sec:upper-bound}.

\subsection{Application: lower bounds for sparse-grid interpolation}
Using similar arguments, analogous lower bounds follow for sparse-grid
polynomial approximation~$\mathscr{I}_{\Lambda}$ based on Gauss--Hermite
interpolation. The approximation $\mathscr{I}_{\Lambda}$ is defined
by 
\begin{align*}
\mathscr{I}_{\Lambda} & :=\sum_{\boldsymbol{\ell}\in\Lambda}(\mathscr{I}_{\ell_{1}}^{\mathrm{uni}}-\mathscr{I}_{\ell_{1}-1}^{\mathrm{uni}})\otimes\dotsb\otimes(\mathscr{I}_{\ell_{d}}^{\mathrm{uni}}-\mathscr{I}_{\ell_{d}-1}^{\mathrm{uni}})\quad\text{with }\mathscr{I}_{0}^{\mathrm{uni}}=0,
\end{align*}
where $\mathscr{I}_{\ell}^{\mathrm{uni}}$ is the polynomial interpolation
operator using the zeros of $H_{n_{\ell}+1}$.

\begin{theorem}Suppose that a finite index set $\Lambda\subset\mathbb{N}^{d}$
is downward-closed in one direction. Let $\mathscr{I}_{\Lambda}$
be the corresponding sparse-grid approximation, and let $N$ be the
total number of the point evaluations used in $\mathscr{I}_{\Lambda}$.
Then, we have 
\[
\sup_{f\neq0}\frac{\|f-\mathscr{I}_{\Lambda}(f)\|_{L_{\rho}^{1}}}{\|f\|_{H_{\rho}^{\alpha}}}\geq c_{\alpha,d}N^{-\alpha/2}.
\]
For the isotropic index set $\Lambda_{L}$, under the assumptions
of \cref{thm:lower-isotropic-Smolyak}, we have
\[
\sup_{f\neq0}\frac{\|f-\mathscr{I}_{\Lambda_{L}}(f)\|_{L_{\rho}^{1}}}{\|f\|_{H_{\rho}^{\alpha}}}\geq c_{\alpha,d}N^{-\alpha/2}(\ln N)^{\frac{a}{2}(d-1)}.
\]
\end{theorem} 
\begin{proof}
Without loss of generality, we assume that $\Lambda$ is downward-closed
in the first coordinate. Then, for some $\overline{\ell}_{1}\in\mathbb{N}$
we have $(\ell_{1},1\dots,1)\in\Lambda$ for all $\ell_{1}=1,\dots,\overline{\ell}_{1}$.
For $h_{\Lambda}$ defined in \eqref{eq:def-fooling} it follows that
\begin{align*}
\mathscr{I}_{\Lambda}(h_{\Lambda}) & =\mathscr{I}_{\overline{\ell}_{1}}^{\mathrm{uni}}(p_{n}).
\end{align*}
Then, $\|h_{\Lambda}-\mathscr{I}_{\Lambda}(h_{\Lambda})\|_{L_{\rho}^{1}(\mathbb{R}^{d})}=\|p_{n}-\mathscr{I}_{\overline{\ell}_{1}}^{\mathrm{uni}}(p_{n})\|_{L_{\rho}^{1}(\mathbb{R})},$
and since $\mathscr{I}_{\overline{\ell}_{1}}^{\mathrm{uni}}$ is interpolatory,
\begin{align*}
\biggl|
    \int_{\mathbb{R}}p_{n}(x)\rho(x)\mathrm{d}x-Q_{\overline{\ell}_{1}}^{\mathrm{uni}}(p_{n})
\biggr|&= \biggl|
    \int_{\mathbb{R}}\bigl[p_{n}(x)-\mathscr{I}_{\overline{\ell}_{1}}^{\mathrm{uni}}(p_{n})(x)\bigr]\rho(x)\mathrm{d}x
\biggr|\\
&\leq\|p_{n}-\mathscr{I}_{\overline{\ell}_{1}}^{\mathrm{uni}}(p_{n})\|_{L_{\rho}^{1}(\mathbb{R})}.
\end{align*}
Hence, using $\|h_{\Lambda}\|_{H_{\rho}^{\alpha}(\mathbb{R}^{d})}=\left(\sum_{r=0}^{\alpha}\|D^{r}p_{n}\|_{L_{\rho}^{2}(\mathbb{R})}^{2}\right)^{\frac{1}{2}}$,
the argument reduces to bounding the worst-case error for the integration
problem of a univariate algorithm: 
\begin{align*}
\frac{\bigl|\int_{\mathbb{R}}p_{n}(x)\rho(x)\mathrm{d}x-Q_{\overline{\ell}_{1}}^{\mathrm{uni}}(p_{n})\bigr|}{\left(\sum_{r=0}^{\alpha}\|D^{r}p_{n}\|_{L_{\rho}^{2}(\mathbb{R})}^{2}\right)^{\frac{1}{2}}} & \leq\frac{\|h_{\Lambda}-\mathscr{I}_{\Lambda}(h_{\Lambda})\|_{L_{\rho}^{1}(\mathbb{R}^{d})}}{\|h_{\Lambda}\|_{H_{\rho}^{\alpha}(\mathbb{R}^{d})}}.
\end{align*}
Following the arguments in the proofs of \cref{thm:GH-lower-bound-general,thm:lower-isotropic-Smolyak}, the statement follows. 
\end{proof}

\subsection{Upper bound\label{sec:upper-bound}}

In this section, we use the index set 
\[
\Lambda_{L}:=\{\boldsymbol{\ell}\in\mathbb{N}^{d}\mid\,|\rev{\boldsymbol{\ell}|} \leq L\}.
\]
As mentioned earlier, this choice corresponds to the classical sparse-grid
quadrature, also known as the isotropic sparse grid or the classical
Smolyak sparse grid.

To derive an upper bound, we will invoke the general theory of Smolyak
algorithm \cite{Wasilkowski_Wozniakowski_1995Explicitbound}; see
also Smolyak's foundational work \cite{Smolyak.SA_1963_QuadratureInterpolationFormulas}.
The theory therein is for general approximation problem of tensor
product of continuous linear functionals defined on the tensor product
of a reproducing kernel Hilbert space. Although such theory can be
applied here in our setting through the norm-equivalence with the
Hermite spaces, such a strategy will induce an extra constant possibly
dependent on $d$. Instead, we will directly verify that an analogous
result holds in $H_{\rho}^{\alpha}$. For notational convenience,
define $e(S_{\Lambda}(f)):=\sup_{f\neq0}\frac{|I(f)-S_{\Lambda}(f)|}{\|f\|_{H_{\rho}^{\alpha}}}.$
The following result is essentially by Wasilkowski and Woźniakowski
\cite{Wasilkowski_Wozniakowski_1995Explicitbound}.

\begin{proposition}\label{prop:SmolyakWasWos}
	If the worst-case error of $Q_{n_{\ell}}^{\mathrm{uni}}$ for the univariate problem is bounded as
\[
\sup_{0\neq g\in H_{\rho}^{\alpha}}\frac{\bigl|\int_{\mathbb{R}}g(x)\rho(x)\mathrm{d}x-Q_{n_{\ell}}^{\mathrm{uni}}(g)\bigr|}{\|g\|_{H_{\rho}^{\alpha}}}\leq C_{1}\,M^{-\alpha\ell/2}
\]
for all $\ell\ge0$ with $M>1$ independent of $\ell$, 
then, for the multivariate case, the worst-case error
of $S_{\Lambda_{L}}$ is bounded as
\[
e(S_{\Lambda_{L}}(f))\leq C_{2}\,M^{-\alpha L/2}\binom{L}{d-1}
\]
with 
$C_{2}:=C_{1}\Bigl(
		\max\bigl\{
				M^{\alpha/2}\, , \,
				\rev{C_1(1+M^{\alpha/2})}
		\bigr\}\Bigr)^{d-1}$.

\end{proposition} 
\begin{proof}
We closely follow the argument in \cite[Proof of Lemma 2]{Wasilkowski_Wozniakowski_1995Explicitbound}.
The same reasoning applies if %
\rev{bounded linear functionals on $H_{\rho}^{\alpha} (\mathbb{R}^s)$ with $\rho(\bsx)=\frac{\mathrm{e}^{-|\boldsymbol{x}|_2^{2}/2}}{(2\pi)^{s/2}}$ of the form $T_1\otimes\dotsb\otimes T_s$, 
where each $T_j\colon H_{\rho}^{\alpha}(\mathbb{R})\to\mathbb{R}$ is a bounded linear funcitonal on $H_{\rho}^{\alpha}(\mathbb{R})$ with $\rho(x)=\frac{\mathrm{e}^{-|x|^{2}/2}}{(2\pi)^{1/2}}$, 
satisfies $\|T_1\otimes\dotsb \otimes T_s\|_{H_{\rho}^{\alpha}(\mathbb{R}^{s})\to\mathbb{R}}=\prod_{j=1}^s\|T_j\|_{H_{\rho}^{\alpha}(\mathbb{R})\to\mathbb{R}}$ for $s=1,\dots,d$.} 
\rev{Since such an equality holds as per \cref{subsec:tensorised-operators}},
the result follows. 
\end{proof}

Then, the final piece for deriving the numerical integration error
is an error analysis in the space that gives $H_{\rho}^{\alpha}$
upon taking the tensor product. Such result is obtained in our previous
paper \cite{Kazashi.Y_Suzuki_Goda_2023_SuboptimalityGaussHermite}.

\begin{proposition} Suppose that the Gauss--Hermite rule $Q_{\ell}^{\mathrm{uni}}$
for univariate functions uses $n_{\ell}$ quadrature points for $\ell\geq1$.
Suppose further that for some $M_{0}>0$ and $M>1$ we choose $n_{\ell}$
as 
\[
M^{\ell-1}\leq n_{\ell}\leq M_{0}(M^{\ell}-1).
\]
Let $N$ be the resulting quadrature points used by $S_{\Lambda_{L}}$
with $L\geq d+1$. Then, we have 
\[
|I(f)-S_{\Lambda_{L}}(f)|\leq c_{\alpha,d}\|f\|_{H_{\rho}^{\alpha}}(\ln N)^{(d-1)(1+\frac{\alpha}{2})}N^{-\alpha/2},
\]
where the constant $c_{\alpha,d}$ is independent of $f$. \end{proposition} 
\begin{proof}
From \cite[Proposition 3.4]{Kazashi.Y_Suzuki_Goda_2023_SuboptimalityGaussHermite},
we know that the worst-case error in the Sobolev space for univariate
functions is given by 
\[
\sup_{g\neq0}\frac{\big|\int_{\mathbb{R}}g(x)\rho(x)\mathrm{d}x-Q_{n_{\ell}}^{\mathrm{uni}}(g)\big|}{\|g\|_{H_{\rho}^{\alpha}}}\leq c_{\alpha}n_{\ell}^{-\alpha/2}\leq C_{1}\,M^{-\alpha\ell/2}
\]
with\rev{ a constant $c_{\alpha}$ depending on $\alpha$, and} $C_{1}:=c_{\alpha}M^{\alpha/2}$. This bound enables us to invoke
\cref{prop:SmolyakWasWos}, yielding 
\[
e(S_{\Lambda_{L}}(f))\leq C_{2}\,M^{-\alpha L/2}\binom{L}{d-1}
\]
with $C_{2}:=C_{1}\Big(\max\Big\{ M^{\alpha/2},c_{\alpha}M^{\alpha/2}+c_{\alpha}M^{\alpha}\Big\}\Big)^{d-1}$.

To express the right-hand side above in terms of $N$, we obtain some
bounds for powers of $M$. The lower bound $M^{L-d}\leq n_{L-d+1}\leq N$
can be obtained by considering the contribution for the quadrature
corresponding to the index $(\ell_{1},\dots,\ell_{d})$ with $\ell_{1}=L-d+1$
and $\ell_{2}=\dotsb=\ell_{d}=1$. On the other hand, following the
calculations in \cite[Section 5]{Wasilkowski_Wozniakowski_1995Explicitbound}
we get $N\leq M_{0}^{d}\frac{M}{M-1}M^{L}\binom{L-1}{d-1}$. From
$\binom{L-1}{d-1}=\big(1-\frac{d-1}{L}\big)\binom{L}{d-1}\leq\binom{L}{d-1}$
\[
M^{-L}\leq N^{-1}M_{0}^{d}\frac{M}{M-1}\binom{L}{d-1}.
\]
Hence, together with $\binom{L}{d-1}\leq\frac{L^{d-1}}{(d-1)!}$ we
have 
\begin{align*}
e(S_{\Lambda}(f)) & \leq C_{2}\,M^{-\alpha L/2}\binom{L}{d-1} \leq C_{2}\,N^{-\alpha/2}M_{0}^{d\alpha/2}\left(\frac{M}{M-1}\right)^{\alpha/2}\binom{L}{d-1}^{1+\alpha/2}\\
 & \leq\frac{C_{2}}{((d-1)!)^{1+\alpha/2}}\left(\frac{M_{0}^{d}M}{M-1}\right)^{\alpha/2}N^{-\alpha/2}L^{(d-1)(1+\alpha/2)}.
\end{align*}
To bound the factor of $L$ in terms of $\ln N$, note that $2\leq L-d+1$
implies $M\leq n_{2}\leq N$; combined with $M^{L-d}\leq N$, this
gives $M^L\le M^dN\le N^{1+d}$ and thus  $L\leq
	{(}\ln N {)}(1+d)/\ln M$. Therefore, we obtain
\[
e(S_{\Lambda_{L}}(f))\leq\frac{C_{2}((1+d)/\ln M)^{(d-1)(1+\alpha/2)}}{((d-1)!)^{1+\alpha/2}}\left(\frac{M_{0}^{d}M}{M-1}\right)^{\alpha/2}\,N^{-\alpha/2}(\ln N)^{(d-1)(1+\alpha/2)}.
\]
Hence, the statement follows with $c_{\alpha,d}:=\frac{C_{2}((1+d)/\ln M)^{(d-1)(1+\alpha/2)}}{((d-1)!)^{1+\alpha/2}}\left(\frac{M_{0}^{d}M}{M-1}\right)^{\alpha/2}$. 
\end{proof}

\section{Quasi-Monte Carlo methods}

\label{sec:QMC} 
A quasi-Monte Carlo method is an equal-weight integration
rule. It approximates the integral of a function $g$ by
\[
\frac{1}{N}\sum_{j=1}^{N}g(\boldsymbol{t}^{(j)})\approx\int_{[0,1]^{d}}g(\boldsymbol{x})\mathrm{d}\boldsymbol{x}
\]
with nodes $\boldsymbol{t}^{(j)}\in[0,1]^{d}$, $j=1,\dots,N$. 
Popular classes for the choice
of $\boldsymbol{t}_{j}$ include lattice points \rev{\cite{Niederreiter.H_1992_book,Sloan.I.H_Joe_1994_book,Dick.J_Kritzer_Pillichshammer_2022_Book_LatticeRules}},
\rev{(}higher-order\rev{)} digital nets \rev{\cite{Niederreiter.H_1992_book,Dick.J_Pillichshammer_2010_book}},
and low discrepancy points/sequences \rev{(such as Halton and Kronecker sequences)} \cite[Chapter~3]{Niederreiter.H_1992_book}.
These classes are not mutually exclusive:  lattice points and 
digital nets may also have low discrepancy; 
see, for instance, \cite[Chapter~5]{Dick.J_Kritzer_Pillichshammer_2022_Book_LatticeRules} and \cite[Chapter~5]{Dick.J_Pillichshammer_2010_book}.

An important point to mention is that if one views $\{\boldsymbol{t}_{j}\}$
as a low-discrepancy point set and uses the Koksma--Hlawka inequality
as a framework to bound the QMC error, then the best achievable rate
is linear convergence. This is because the Koksma--Hlawka inequality
is based on first-order derivatives. To achieve higher-order convergence,
one needs a framework that exploits higher-order derivatives.

QMC points are typically defined on the unit cube $[0,1]^{d}$. To
integrate functions over $\mathbb{R}^{d}$, we therefore apply domain truncation and a change
of variables: 
\begin{align}
\int_{\mathbb{R}^{d}}f(\boldsymbol{x})\,\rho(\boldsymbol{x})\,\mathrm{d}\boldsymbol{x} & =
\int_{\Omega}f(\boldsymbol{x})\,\rho(\boldsymbol{x})\,\mathrm{d}\boldsymbol{x}+\int_{\mathbb{R}^{d}\setminus \Omega}f(\boldsymbol{x})\,\rho(\boldsymbol{x})\,\mathrm{d}\boldsymbol{x}
\label{eq:QMC-domain-decomp}
\end{align}
and
\begin{align}
\int_{\Omega}f(\boldsymbol{x})\,\rho(\boldsymbol{x})\,\mathrm{d}\boldsymbol{x}&=
\int_{\rev{(}0,1\rev{)}^{d}}f(\boldsymbol{\Psi}(\boldsymbol{t}))\,\rho(\boldsymbol{\Psi}(\boldsymbol{t}))\,|\det(D\boldsymbol{\Psi}(\boldsymbol{t}))|\,\mathrm{d}\boldsymbol{t}\nonumber\\
 & \approx\frac{1}{N}\sum_{j=1}^{N}f\bigl(\boldsymbol{\Psi}(\boldsymbol{t}^{(j)})\bigr)\rho\bigl(\boldsymbol{\Psi}(\boldsymbol{t}^{(j)})\bigr)|\det(D\boldsymbol{\Psi}(\boldsymbol{t}^{(j)}))|=:Q_{N}(f).\label{eq:QMC-CV}
\end{align}
Here, $\Omega\subseteq\mathbb{R}^d$ is a suitable integration domain, possibly $\Omega=\mathbb{R}^d$\rev{, and $\boldsymbol{\Psi}:(0,1)^d\to \Omega$ is a suitable mapping.}  
Often the mapping is component-wise, $\boldsymbol{\Psi}(\boldsymbol{t})\!\!\;=\!\!\;(\Psi_{1}(t_{1}),\dots,\Psi_{d}(t_{d}))$,
and non-negative, in which case the Jacobian factor simplifies as
\[
|\det(D\boldsymbol{\Psi}(\boldsymbol{t}))|=\prod_{k=1}^{d}\Psi_{k}'(t_{k}).
\]
In this paper, we consider two choices for $\Psi$: an affine map
and a Möbius transformation\rev{, namely the cotangent transform}. 
For the former, $\Omega$ is a suitable box; for the latter, $\Omega$ is the whole $\mathbb{R}^d$.
We will show that both of them map $f\in H_{\rho}^{\alpha}$
to Sobolev spaces over the unit cube for which QMC methods are known
to work well.

As $\{\boldsymbol{t}^{(1)},\ldots,\boldsymbol{t}^{(N)}\}$ we consider
\emph{good rank-$1$ lattices} and \emph{digital nets}. A rank-$1$
lattice is defined by 
\[
\boldsymbol{t}^{(j)}:=\bigg\{\frac{\boldsymbol{z}j}{N}\bigg\}\;\;\text{ for }j=1,\ldots,N,
\]
where $\{\cdot\}$ expresses the fractional part of each component, and the integer
vector $\boldsymbol{z}\in\mathbb{N}^{d}$ is called a generating vector.
A good generating vector can be found efficiently using a greedy algorithm
known as the component-by-component (CBC) construction, which provably
yields a vector with a satisfactorily small worst-case error in various
function spaces \cite{Dick.J_Kritzer_Pillichshammer_2022_Book_LatticeRules}.

A digital net is defined as follows. For a fixed prime $p$ and a
positive integer $m$, we denote the $p$-adic expansion of a non-negative
integer $\rev{j}<p^{m}$ and a real number $x\in[0,1)$ by 
\[
\rev{j=\iota_{0}+\iota_{1}p+\cdots + \iota_{m-1}p^{m-1}}\qquad  \text{and} \qquad 
x=\frac{\xi_{1}}{p}+\frac{\xi_{2}}{p^{2}}+\cdots,
\]
respectively. \rev{Here, the digits $\iota_k$ and $\xi_k$ are in $\ZZ_p$, the set $\{0,1,\dots,p-1\}$ equipped with the usual arithmetic
operations modulo $p$.} We then write $\vec{j}=(\iota_{0},\iota_{1},\ldots,\iota_{m-1})^{\top}\in\ZZ_{p}^{m}$
and $\vec{x}=(\xi_{1},\xi_{2},\ldots)^{\top}\in\ZZ_{p}^{\NN}$. 
If $x$ admits two different $p$-adic expansions, we choose the one
in which infinitely many of the digits $\xi_{k}$ differ from $p-1$.

Now, for suitably designed generating matrices $C_{1},\ldots,C_{d}\in\ZZ_{p}^{\NN\times m}$, \rev{
consider the corresponding digit vectors $\vec{t}_{j,i}\in \mathbb{Z}^\mathbb{N}_p$ obtained by the matrix-vector product modulo $p$:
\[
\vec{t}_{j,i}=C_{i}\,\vec{j}\;\;\text{ for }i=1,\ldots,d\text{ and }j=0,1,\ldots,N-1.
\]
Then, the digital net consisting of $N=p^{m}$ points corresponding to $C_1,\dots,C_d$ is defined as the set $\{\bst^{(0)},\ldots,\bst^{(N-1)}\}$, where each component $t_{j,i}$ of the point $\bst^{(j)}=(t_{j,1},\ldots,t_{j,d})\in [0,1]^d$ is determined by the corresponding digit vector $\vec{t}_{j,i}$. 
}%
Note that if every column of the generating matrices has only finitely
many non-zero entries, this matrix-vector product is well-defined.
The classical (first-order) digital nets constructed by Sobol’ \cite{Sobol1967},
Niederreiter \cite{Niederreiter1988}, and Niederreiter--Xing \cite{NX96} are included
as special cases of this construction. In \cite{Dick2008}, Dick introduced
the notion of \emph{higher-order digital nets}, which also form a
special class within this framework. An order-$\alpha$ digital net
with $\alpha\in\NN$, $\alpha\ge2$, is constructed using generating
matrices $C_{1},\ldots,C_{d}$, which are themselves derived from
the generating matrices $D_{1},\ldots,D_{\alpha d}$ for first-order
digital nets in dimension $\alpha d$.
We refer the reader to \cite{GS2020} for a review of higher-order nets.

It is worth noting that QMC methods presented in this section always have positive weights, i.e., when seeing the quadrature \eqref{eq:QMC-CV}  as $Q_N(f)=\sum_{j=1}^N w_j f(\bsx_j) $, weights $w_j$'s are all positive. 
As such, nonnegative integrands yield nonnegative approximated integral values. This nonnegativity-preserving property is important in situations where the quantity of interest must be nonnegative, for example probabilities. 
Positive quadrature weights are also useful when discretizing optimization problems involving expectations \cite{Martin.M_Nobile_2021_PDEConstrainedOptimalControl} with convex objective functions, since they preserve convexity. 
By contrast, other quadrature methods may involve negative weights and thus lose these desirable features.

\rev{Before going into details, we first summarize our main results of this section in the following theorem.

\begin{theorem}\label{thm:qmc-summary}The quasi-Monte Carlo rule $Q_{N}$ \eqref{eq:QMC-CV}
achieves the error bound 
\begin{align*}
\sup_{f\neq0}\frac{|I(f)-Q_{N}(f)|}{\|f\|_{H_{\rho}^{\alpha}}} & \leq C_{d,\alpha}\frac{(\ln N)^{\gamma}}{N^{\alpha}},
\end{align*}
where the exponent $\gamma$ depends on the QMC construction: 
\[
\gamma=\begin{cases}
\frac{3\alpha d}{2}+\frac{d}{4} & \text{(Rank-1 lattice rule with affine map)}\\
\frac{\alpha d}{2}+\frac{3d}{4}-\frac{1}{2} & \text{(Higher order digital net with affine map)}\\
d\alpha & \text{(Rank-1 lattice rule with cotangent transform)}\\
(d-1)/2 & \text{(Higher order digital net with cotangent transform)}
\end{cases}.
\]
The polynomial convergence rate $\alpha$ is optimal in $H_{\rho}^{\alpha}$:
no algorithm can achieve a better polynomial rate. Furthermore, the
exponent $(d-1)/2$ in the logarithmic factor is also optimal. \end{theorem}

The bound for higher-order digital nets with an affine map is due to Dick et al.~\cite{Dick.J_Irrgeher_Leobacher_Pillichshammer_2018_SINUM_Hermite}. 
In the following subsections, we present detailed algorithms and establish the remaining upper bounds. The optimality statements then follow from the lower bound in \cref{thm:lower-all-algorithm}.
}
\subsection{Affinely-mapped QMC}

For $\boldsymbol{b}=(b,\ldots,b)\in(0,\infty)^{d}$, define the isotropic
affine transformation $\boldsymbol{\Psi}_{\mathrm{affine},b}:[0,1]^{d}\to\prod_{k=1}^{d}[-b,b]\eqqcolon[-\bsb,\bsb]\subset\mathbb{R}^{d}$
by 
\[
\boldsymbol{\Psi}_{\mathrm{affine}}:=\boldsymbol{\Psi}_{\mathrm{affine},b}(\boldsymbol{t})=2b\boldsymbol{t}-\boldsymbol{b}.
\]
Choosing $[-\boldsymbol{b},\boldsymbol{b}]$ as $\Omega$ in \eqref{eq:QMC-domain-decomp} and using $\boldsymbol{\Psi}_{\mathrm{affine}}$ in \eqref{eq:QMC-CV}
yields
\begin{equation}
Q_{N}(f)=\frac{(2b)^{d}}{N}\sum_{j=1}^{N}f\bigl(\boldsymbol{\Psi}_{\mathrm{affine}}(\boldsymbol{t}^{(j)})\bigr)\rho\bigl(\boldsymbol{\Psi}_{\mathrm{affine}}(\boldsymbol{t}^{(j)})\bigr),\quad f\in H_{\rho}^{\alpha}.\label{eq:affine-QMC}
\end{equation}
The value $b$ controls the truncation error, namely the integral of $f\rho$ over $\mathbb{R}^d\setminus [-\boldsymbol{b},\boldsymbol{b}]$, whereas $Q_N$ controls the numerical integration error. 
With a suitable choice of $b$ and point set $\{\boldsymbol{t}_{1},\ldots,\boldsymbol{t}_{N}\}$,
QMC methods attain the optimal convergence rate up to logarithmic
factors, with the dominant part $N^{-\alpha}$ being optimal.
\begin{proposition}[Affinely-mapped lattice]\label{prop:lattice-affin}
Let $\alpha\in\NN$, and let $Q_{N}$ be the rank-$1$ lattice rule
defined in \eqref{eq:affine-QMC} with $b=\rev{(2+\eta)}\sqrt{\alpha\ln N}$\rev{, where $\eta>0$ is arbitrary}. 
Then, with a generating vector $\boldsymbol{z}$ obtained by the CBC
construction in \cite[Algorithm~3.14]{Dick.J_Kritzer_Pillichshammer_2022_Book_LatticeRules},
$Q_{N}$ satisfies 
\begin{equation}
\sup_{f\neq0}\frac{|I(f)-Q_{N}(f)|}{\|f\|_{H_{\rho}^{\alpha}}}\leq C_{\alpha,d,\rev{\eta}}\;\!N^{-\alpha}(\ln N)^{\frac{3\alpha d}{2}+\frac{d}{4}}.\label{eq:aff-lattice-error}
\end{equation}

\end{proposition}
\begin{proof}
First, we will show that $f\in H_{\rho}^{\alpha}$ implies $\|f\rho\|_{\text{unanchored}}+\|f\rho\|_{\text{decay}\rev{, \varepsilon}}\leq C_{d,\alpha,\rev{\varepsilon}}\|f\|_{H_{\rho}^{\alpha}(\mathbb{R}^{d})}$ \rev{for every  $\varepsilon\in(0,1/2)$ with} 
 some constant $C_{d,\alpha\rev{,\varepsilon}}>0$, where 
\[
\|f\rho\|_{\text{decay}\rev{, \varepsilon}}\coloneqq\sup_{\substack{\bsx\in\R^{d}\\
\mathbf{r}\in\{0,\ldots,\alpha-1\}^{d}
}
}\left|\rho^{-\varepsilon}(\bsx)\,D^{\mathbf{r}}(f\rho)(\bsx)\right|<\infty, 
\]
and 
\[
\|f\rho\|_{\text{unanchored}}\coloneqq\sup_{[\bsa,\bsb]\subset\R^{d}}\|f\rho\|_{\mathscr{H}^{\alpha}([\bsa,\bsb])}<\infty
\]
with 
\begin{align}
\langle f,h\rangle_{\mathscr{H}^{\alpha}([\boldsymbol{a},\boldsymbol{b}])} & := \sum_{\substack{\mathbf{r}\in\{0,\dots,\alpha\}^{d}\\
\mathfrak{v}:=\{k:r_{k}=\alpha\}
}
}\int_{[\boldsymbol{a}_{\mathfrak{v}},\boldsymbol{b}_{\mathfrak{v}}]}\left(\int_{[\boldsymbol{a}_{-\mathfrak{v}},\boldsymbol{b}_{-\mathfrak{v}}]}D^{\mathbf{r}}f(\boldsymbol{x})\,\mathrm{d}\boldsymbol{x}_{-\mathfrak{v}}\right) \label{eq:def-unanchored-alpha-ab} \\
& \qquad \qquad \qquad \qquad \times \left(\int_{[\boldsymbol{a}_{-\mathfrak{v}},\boldsymbol{b}_{-\mathfrak{v}}]}D^{\mathbf{r}}h(\boldsymbol{x})\,\mathrm{d}\boldsymbol{x}_{-\mathfrak{v}}\right)\mathrm{d}\boldsymbol{x}_{\mathfrak{v}} \notag
\end{align}
for bounded $[\bsa,\bsb]$. 
\rev{The condition $\|f\rho\|_{\text{decay}, \varepsilon}<\infty$ corresponds to the exponential decay condition \cite[Equation (22)]{Nuyens.D_Suzuki_2022_ScaledLatticeRd},  whereas $\|f\rho\|_{\text{unanchored}}<\infty$ corresponds to the requirement that \cite[Equation (24)]{Nuyens.D_Suzuki_2022_ScaledLatticeRd} be  finite.} 
In the proof we use the following identities
\begin{align*}
D^{\mathbf{r}}(H_{\mathbf{k}}\,\rho) & =(-1)^{|\mathbf{r}|}\sqrt{\frac{(\mathbf{k}+\mathbf{r})!}{\mathbf{k}!}}H_{\mathbf{k}+\mathbf{r}}\,\rho,\ \text{ and }\ D^{\mathbf{r}}\rho=(-1)^{|\mathbf{r}|}\sqrt{\mathbf{r}!}\,H_{\mathbf{r}}\,\rho,
\end{align*}
which can be derived from the definition of Hermite polynomial.

For $g=f\rho$ and $\mathbf{r}\in\mathbb{N}_{0}^{d}$ with $|\mathbf{r}|_{\infty}\leq\alpha$,
the product rule yields 
\begin{align}
D^{\mathbf{r}}g & =\sum_{\mathbf{s}\le\mathbf{r}}\binom{\mathbf{r}}{\mathbf{s}}(D^{\mathbf{s}}\rho)\,(D^{\mathbf{r}-\mathbf{s}}f)=\sum_{\mathbf{s}\le\mathbf{r}}(-1)^{|\mathbf{s}|}\sqrt{\mathbf{s}!}\binom{\mathbf{r}}{\mathbf{s}}\,H_{\mathbf{s}}\,\rho\,D^{\mathbf{r}-\mathbf{s}}f.\label{eq:Leibniz}
\end{align}

Now, 
\begin{equation}
\sup_{[\boldsymbol{a},\boldsymbol{b}]\subset\mathbb{R}^{d}}\|g\|_{\mathscr{H}^{\alpha}([\boldsymbol{a},\boldsymbol{b}])}^{2}\leq\sum_{\substack{\mathbf{r}\in\{0,\ldots,\alpha\}^{d}\\
\mathfrak{v}:=\left\{ j:r_{j}=\alpha\right\} 
}
}\int_{\mathbb{R}^{|\mathfrak{v}|}}\left(\int_{\mathbb{R}^{d-|\mathfrak{\mathfrak{v}}|}}|D^{\mathbf{r}}g(\boldsymbol{x})|\,\mathrm{d}\boldsymbol{x}_{-\mathfrak{v}}\right)^{2}\mathrm{d}\boldsymbol{x}_{\mathfrak{v}}.\label{eq:ab-bound-1}
\end{equation}
Using the Cauchy--Schwarz inequality, the inner-integral can be bounded
as 
\begin{align*}
&\int_{\mathbb{R}^{d-|\mathfrak{\mathfrak{v}}|}}  |D^{\mathbf{r}}g(\boldsymbol{x})|\,\mathrm{d}\boldsymbol{x}_{-\mathfrak{v}}
\leq
\sum_{\mathbf{s}\le\mathbf{r}}\sqrt{\mathbf{s}!}\binom{\mathbf{r}}{\mathbf{s}}\int_{\mathbb{R}^{d-|\mathfrak{\mathfrak{v}}|}}|H_{\mathbf{s}}(\boldsymbol{x})D^{\mathbf{r}-\mathbf{s}}f(\boldsymbol{x})|\,\rho(\boldsymbol{x})\mathrm{d}\boldsymbol{x}_{-\mathfrak{v}}\\
&\ \ \kern-3mm \leq\sum_{\mathbf{s}\le\mathbf{r}}\sqrt{\mathbf{s}!}\binom{\mathbf{r}}{\mathbf{s}}|H_{\boldsymbol{s}(\mathfrak{v})}(\boldsymbol{x}_{\mathfrak{v}})|\kern1.5pt
	\rho_{\mathfrak{\mathfrak{v}}}(\boldsymbol{x}_{\mathfrak{v}})
	\kern1pt
{\|H_{\boldsymbol{s}(-\mathfrak{v})}\|_{L^{2}_{\rho_{\mathfrak{-\mathfrak{v}}}}} \kern-1mm }\left(\int_{\mathbb{R}^{d-|\mathfrak{\mathfrak{v}}|}}\kern-2mm |D^{\mathbf{r}-\mathbf{s}}f(\boldsymbol{x})|^{2}\,\rho_{\mathfrak{-\mathfrak{v}}}(\boldsymbol{x}_{-\mathfrak{v}})\mathrm{d}\boldsymbol{x}_{-\mathfrak{v}}\right)^{\kern-1mm 1/2}
\end{align*}
with %
$\|H_{\boldsymbol{s}(-\mathfrak{v})}\|_{L^{2}_{\rho_{\mathfrak{-\mathfrak{v}}}}}:=\left(\int_{\mathbb{R}^{d-|\mathfrak{\mathfrak{v}}|}}|H_{\boldsymbol{s}(-\mathfrak{v})}(\boldsymbol{x}_{-\mathfrak{v}})|^{2}\,\rho_{\mathfrak{-\mathfrak{v}}}(\boldsymbol{x}_{-\mathfrak{v}})\mathrm{d}\boldsymbol{x}_{-\mathfrak{v}}\right)^{1/2}$.
Here, we used the notations $H_{\boldsymbol{s}(\mathfrak{v})}(\rev{\boldsymbol{x}}_{\mathfrak{v}}):=\prod_{j\in\mathfrak{v}}H_{s_{j}}(x_{j})$,
$H_{\boldsymbol{s}(-\mathfrak{v})}(\rev{\boldsymbol{x}}_{-\mathfrak{v}}):=\prod_{j\in\{1,\dots,d\}\setminus\mathfrak{v}}H_{s_{j}}(x_{j})$,
and similarly $\rho_{\mathfrak{v}}\rev{(\boldsymbol{x_{\mathfrak{v}}})}:=\prod_{j\in\mathfrak{v}}\rho(x_{j})$
and $\rho_{-\mathfrak{v}} \rev{(\boldsymbol{x_{-\mathfrak{v}}})}:=\prod_{j\in\{1,\dots,d\}\setminus\mathfrak{v}}\rho(x_{j})$.
Noting $\|H_{\boldsymbol{s}(-\mathfrak{v})}\|_{L^{2}_{\rho_{\mathfrak{-\mathfrak{v}}}}}=1$, it
follows 
\begin{align*}
\int_{\mathbb{R}^{|\mathfrak{\mathfrak{v}}|}} & \left(\int_{\mathbb{R}^{d-|\mathfrak{\mathfrak{v}}|}}|D^{\mathbf{r}}g(\boldsymbol{x})|\,\mathrm{d}\boldsymbol{x}_{-\mathfrak{v}}\right)^{2}\mathrm{d}\boldsymbol{x}_{\mathfrak{v}}\\
 & \leq C_{d,\rev{\alpha}}\sum_{\mathbf{s}\le\mathbf{r}}\mathbf{s}!\binom{\mathbf{r}}{\mathbf{s}}^{2}\int_{\mathbb{R}^{d}}|H_{\boldsymbol{s}(\mathfrak{v})}(\boldsymbol{x}_{\mathfrak{v}})|^{2}\rho_{\mathfrak{\mathfrak{v}}}^{2}(\boldsymbol{x}_{\mathfrak{v}})\,|D^{\mathbf{r}-\mathbf{s}}f(\boldsymbol{x})|^{2}\,\rho_{\mathfrak{-\mathfrak{v}}}(\boldsymbol{x}_{-\mathfrak{v}})\mathrm{d}\boldsymbol{x}\\
 & \leq C_{d,\rev{\alpha}}\sum_{\mathbf{s}\le\mathbf{r}}\mathbf{s}!\binom{\mathbf{r}}{\mathbf{s}}^{2}\sup_{\boldsymbol{y}_{\mathfrak{v}}\in\mathbb{R}^{|\mathfrak{v}|}}|H_{\boldsymbol{s}(\mathfrak{v})}(\boldsymbol{y}_{\mathfrak{v}})|^{2}\rho_{\mathfrak{\mathfrak{v}}}(\boldsymbol{y}_{\mathfrak{v}})\int_{\mathbb{R}^{d}}|D^{\mathbf{r}-\mathbf{s}}f(\boldsymbol{x})|^{2}\,\rho(\boldsymbol{x})\mathrm{d}\boldsymbol{x}.
\end{align*}
Therefore, from \eqref{eq:ab-bound-1} 
\[
\rev{\|f\rho\|_{\text{unanchored}}=}\sup_{[\boldsymbol{a},\boldsymbol{b}]\subset\mathbb{R}^{d}}\|g\|_{\mathscr{H}^{\alpha}([\boldsymbol{a},\boldsymbol{b}])}\leq C_{d,\alpha}\|f\|_{H_{\rho}^{\alpha}(\mathbb{R}^{d})}
\]
holds for some constant $C_{d,\alpha}>0$.

To show $\lVert g\rVert_{\text{decay}\rev{, \varepsilon}}\leq C\|f\|_{H_{\rho}^{\alpha}(\mathbb{R}^{d})}$,
we will invoke the Sobolev embedding in \cref{prop:1st-order-embedding}
to bound the supremum norm. Define $h_{\mathbf{r}}:=\rho^{-\varepsilon}D^{\mathbf{r}}g=\rho^{-\varepsilon}D^{\mathbf{r}}(f\rho)$
for $\mathbf{r}\in\{0,\dots,\alpha-1\}^{d}$, for which we will derive
a bound on its Sobolev norm. For $\mathbf{k}\in\{0,1\}^{d}$ 
\[
D^{\mathbf{k}}h_{\mathbf{r}}(\boldsymbol{x})=
\sum_{\mathbf{j}\leq\mathbf{k}}\binom{\mathbf{k}}{\mathbf{j}}
\biggl(
	\prod_{j\in\mathbf{j}
}{\varepsilon\, x_{j}}\biggr)\rho^{-\varepsilon}(\boldsymbol{x})(D^{\mathbf{k}-\mathbf{j}+\mathbf{r}}g(\boldsymbol{x}))
\]
holds. With 
$p_{\mathbf{j}}(\boldsymbol{x}):=\prod_{j\in\mathbf{j}}\varepsilon\,x_{j}$, 
taking the $L^{2}(\mathbb{R}^{d})$-norm of both sides together
with triangular inequality yields 
\begin{align}
 \| D^{\mathbf{k}}h_{\mathbf{r}}\|_{L^{2}(\mathbb{R}^{d})} & \leq\sum_{\mathbf{j}\leq\mathbf{k}}\binom{\mathbf{k}}{\mathbf{j}}\sum_{\mathbf{s}\le\mathbf{r}+\mathbf{k}-\mathbf{j}}\sqrt{\mathbf{s}!}\binom{\mathbf{r}+\mathbf{k}-\mathbf{j}}{\mathbf{s}}\|p_{\mathbf{j}}\rho^{-\varepsilon}\,H_{\mathbf{s}}\,\rho\,D^{\mathbf{r}+\mathbf{k}-\mathbf{j}-\mathbf{s}}f\|_{L^{2}(\mathbb{R}^{d})}\nonumber \\
 & \leq\sum_{\mathbf{j}\leq\mathbf{k}}\binom{\mathbf{k}}{\mathbf{j}}\sum_{\mathbf{s}\le\mathbf{r}+\mathbf{k}-\mathbf{j}}\sqrt{\mathbf{s}!}\binom{\mathbf{r}+\mathbf{k}-\mathbf{j}}{\mathbf{s}}\sup_{\boldsymbol{y}\in\mathbb{R}^{d}}|p_{\mathbf{j}}(\boldsymbol{y})H_{\mathbf{s}}(\boldsymbol{y})\rho^{1/2-\varepsilon}(\boldsymbol{y})|\nonumber \\
 & \qquad\qquad\times\|\,\,\sqrt{\rho}\,(D^{\mathbf{r}+\mathbf{k}-\mathbf{j}-\mathbf{s}}f)\|_{L^{2}(\mathbb{R}^{d})}\nonumber \\
 & =\sum_{\mathbf{j}\leq\mathbf{k}}\binom{\mathbf{k}}{\mathbf{j}}\sum_{\mathbf{s}\le\mathbf{r}+\mathbf{k}-\mathbf{j}}\sqrt{\mathbf{s}!}\binom{\mathbf{r}+\mathbf{k}-\mathbf{j}}{\mathbf{s}}\sup_{\boldsymbol{y}\in\mathbb{R}^{d}}|p_{\mathbf{j}}(\boldsymbol{y})H_{\mathbf{s}}(\boldsymbol{y})\rho^{1/2-\varepsilon}(\boldsymbol{y})|\nonumber \\
 & \qquad\qquad\times\|D^{\mathbf{r}+\mathbf{k}-\mathbf{j}-\mathbf{s}}f\|_{L_{\rho}^{2}(\mathbb{R}^{d})}<\infty,\label{eq:bound-Dk-h}
\end{align}
and we used the identity \eqref{eq:Leibniz} for $D^{\mathbf{k}-\mathbf{j}+\mathbf{r}}g$.
In \eqref{eq:bound-Dk-h}, the multi-index satisfies $\mathbf{r}+\mathbf{k}-\mathbf{j}-\mathbf{s}\in\{0,\dots,\alpha\}^{d}$,
and thus 
\begin{align*}
\sum_{\mathbf{k}\in\{0,1\}^{d}}\|D^{\mathbf{k}}(\rho^{-\varepsilon}D^{\mathbf{r}}g)\|_{L^{2}(\mathbb{R}^{d})}^{2}
=
	\sum_{\mathbf{k}\in\{0,1\}^{d}}\|D^{\mathbf{k}}h_{\mathbf{r}}\|_{L^{2}(\mathbb{R}^{d})}^{2} & \leq C_{\rev{\varepsilon}}\|f\|_{H_{\rho}^{\alpha}(\mathbb{R}^{d})}^{2}.
\end{align*}
Therefore, in view of \cref{prop:1st-order-embedding} we have $\sup_{\boldsymbol{x}\in\mathbb{R}^{d}}|\rho^{-\varepsilon}(\boldsymbol{x})D^{\mathbf{r}}g(\boldsymbol{x})|\leq C_{\rev{\varepsilon}}\|f\|_{H_{\rho}^{\alpha}(\mathbb{R}^{d})}$, hence $\|f\rho\|_{\text{unanchored}}+\|f\rho\|_{\text{decay}\rev{, \varepsilon}}\leq C_{d,\alpha,\rev{\varepsilon}}\|f\|_{H_{\rho}^{\alpha}(\mathbb{R}^{d})}$
follows.

With this, from the first displayed equation in \cite[Lemma~5]{Nuyens.D_Suzuki_2022_ScaledLatticeRd}
together with \cite[Propositions~7 \& 8]{Nuyens.D_Suzuki_2022_ScaledLatticeRd},
if $\boldsymbol{z}$ satisfies the second displayed equation in \cite[Lemma 5]{Nuyens.D_Suzuki_2022_ScaledLatticeRd},
then this $\boldsymbol{z}$ satisfies \eqref{eq:aff-lattice-error};
see the proof of \cite[Theorem 2]{Nuyens.D_Suzuki_2022_ScaledLatticeRd}.
Since such $\boldsymbol{z}$ can be constructed using a CBC construction
in \cite[Algorithm~3.14]{Dick.J_Kritzer_Pillichshammer_2022_Book_LatticeRules}, \rev{we have
\[
|I(f)-Q_N(f)|\le C_{\alpha,d,\beta} \|f\|_{H^{\alpha}_{\rho}(\mathbb{R}^d)}\bigg(\frac{(\ln N)^{\alpha d}}{N^{\alpha}}b^{d(\alpha +1/2)}+ b^{d\alpha} \mathrm{e}^{-\beta b^2 }+b^{d-2}\mathrm{e}^{-\beta b^2 } \bigg),
\]
where 
$\beta:=\varepsilon/2 \in (0,1/4)$ corresponds to  $\beta$ in \cite{Nuyens.D_Suzuki_2022_ScaledLatticeRd}, and 
the constant $C_{\alpha,d,\beta}$ is independent of $N$ and $b$. %
For $\eta>0$ arbitrary, we now choose $b=(2+\eta)\sqrt{\alpha\ln(N)}$. Then, since the bound above holds for any $1/(2+\eta)^2\leq \beta<1/4$, for such $\beta$ this choice of $b$ ensures 
 $\mathrm{e}^{-\beta b^2}=N^{- \beta (2+\eta)^2\alpha}\leq N^{-\alpha}$. This completes the proof.
}
\end{proof}

Affinely mapped digital nets also achieve the optimal convergence
rate, up to a logarithmic factor, as shown in \cite[Corollary 1]{Dick.J_Irrgeher_Leobacher_Pillichshammer_2018_SINUM_Hermite}.
\begin{proposition}[Affinely-mapped digital net, {\cite[Corollary 1]{Dick.J_Irrgeher_Leobacher_Pillichshammer_2018_SINUM_Hermite}}]
Let $\alpha\in\NN$, and let $Q_{N}$ be the quadrature rule defined
in \eqref{eq:affine-QMC} with $b=2\sqrt{{\alpha}\ln N}$, where points
$(\boldsymbol{t}^{(j)})_{j=1,\dots,N}$ are given by a higher-order
net of order $(2\alpha+1)$ with $N=p^{m}$ elements. Then, the worst
case error is bounded by 
\[
\sup_{f\neq0}\frac{|I(f)-Q_{N}(f)|}{\|f\|_{H_{\rho}^{\alpha}}}\leq C_{\alpha,d}N^{-\alpha}(\ln N)^{\frac{\alpha d}{2}+\frac{3d}{4}-\frac{1}{2}}.
\]
\end{proposition}
In the result quoted above, higher-order nets of order $(2\alpha+1)$ are considered. On the other hand, as noted in \cite{Dick.J_Irrgeher_Leobacher_Pillichshammer_2018_SINUM_Hermite}, higher-order nets of order $\alpha$ also achieve the optimal polynomial rate $N^{-\alpha}$, but with a logarithmic factor of a larger exponent.

\subsection{QMC with Möbius transformation}

An alternative to the affine transformation is the change of variables
$\Psi:(0,1)^{d}\to\R^{d}$ defined by 
\begin{equation}\label{def:cot}
\boldsymbol{\Psi}_{\mathrm{cot}}(\bst)\coloneqq(\phi(t_{1}),\ldots,\phi(t_{d})),\;\;\;\phi(t)\coloneqq-\cot(\pi t).
\end{equation}
This is a specific Möbius transformation considered in \cite{Suzuki.Y_etal_2025_MobiustransformedTrapezoidalRule}.
\rev{Using $\boldsymbol{\Psi}_{\mathrm{cot}}$ in \eqref{eq:QMC-CV} yields
	the following formula: 
	\begin{equation}
	Q_{N}(f)=\frac{1}{N}\sum_{j=1}^{N}f\left(\boldsymbol{\Psi}_{\mathrm{cot}}(\boldsymbol{t}^{(j)})\right)\rho\left(\boldsymbol{\Psi}_{\mathrm{cot}}(\boldsymbol{t}^{(j)})\right)\prod_{k=1}^{d}\phi'(t^{(j)}_{k}),\quad f\in H_{\rho}^{\alpha}.\label{eq:QMC+Moebius}
	\end{equation}}%
\rev{It turns out that the transformed integrand, defined on the unit cube by $g:=f(\boldsymbol{\Psi}_{\mathrm{cot}}(\cdot))\rho(\boldsymbol{\Psi}_{\mathrm{cot}}(\cdot))\prod_{k=1}^{d}\phi'$, belongs to suitable Sobolev spaces for which QMC methods work well. To show this, we first derive results on the derivatives of $g$. 
}

\rev{The derivatives of the transformation $\phi(t)=-\cot(\pi t)$ satisfy the following.}
\begin{lemma}\label{lem:cot-polynomial}
    For every $\tau\in\NN$, $\phi^{(\tau)}\circ\phi^{-1}(x)$ is a polynomial in \(x\) of degree $\tau+1$. 
\end{lemma}
\begin{proof}
Noting $\phi^{(\tau)}(t)=-\pi^{\tau}\cot^{(\tau)}(\pi t)$ and $\phi^{-1}(x)=\frac{1}{\pi}\cot^{-1}(-x)$,
it suffices to show that $\cot^{(\tau)}\circ\cot^{-1}(x)$ is a polynomial
of degree $\tau+1$. 
From \cite[Equation~(5)]{Hoffman_1995} we know that $P_{\tau}(x):=\tan^{(\tau)}\circ\tan^{-1}(x)$ is a polynomial of degree $\tau+1$ satisfying \[P_{\tau+1}=\sum_{r=0}^{\tau}\binom{\tau}{r}P_{\tau-r}(x)P_r(x), \]
    with $P_0(x)=x$ and $P_1(x)=x^2+1$. With $\cot(t)=\tan(\pi/2-t)$, the statement follows.
\end{proof}

\rev{Next, we} derive a higher-order chain rule for component-wise composition,
which may be of interest by itself. Although this is a special case
of the general Faa di Bruno formula, component-wise composition makes
the formula simpler and thus easier to work with for broader context
in numerical analysis.

First, we derive a recurrence relation in the univariate setting.
For $r\ge1$ and $1\le \rev{\eta}\le r$, let $B_{r,\rev{\eta}}$ denote the (partial)
Bell polynomial\rev{, which may be defined by
\begin{equation}\label{eq:def-Bell}
	B_{r,\eta}(x_{1},\dots,x_{r})=\sum_{\boldsymbol{0}\le\mathbf{k}\le\boldsymbol{\eta}}r!\bigg(\prod_{\ell=1}^{r}\frac{x_{\ell}^{k_{\ell}}}{k_{\ell}!(\ell!)^{k_{\ell}}}\bigg)\mathbbm{1}_{\{|\mathbf{k}|=\eta,\,\sum_{\ell=1}^{r}\ell k_{\ell}=r\}}(\mathbf{k}),
\end{equation}
where $\mathbbm{1}_{A}(\mathbf{k})$
denotes the indicator function of a set $A$, and $\boldsymbol{\eta}:=(\eta,\ldots,\eta)$; see for example \cite[Definition 11.2]{Charalambides_2002}}. 
We use the convention $B_{0,0}=1$ and $B_{r,0}=0$
for $r\ge1$. 
\rev{We note that although $B_{r,\eta} = B_{r,\eta}(x_1,\dots,x_r)$ is defined as a polynomial in $x_1,\dots,x_r$, the constraints $\sum_{i=1}^r k_i=\eta$ and $\sum_{i=1}^r i k_i=r$ in its definition imply that $B_{r,\eta}$ in fact depends only on $x_1,\dots,x_{r-\eta+1}$.} 

\rev{
We will use the following identity, which shows that $\partial_j B_{r,\eta}$ is, up to a binomial coefficient,  again a partial Bell polynomial. Although the result appears to be well known, a proof seems difficult to find in the literature, so we provide one.
\begin{lemma}\label{lem:bell-der} Let $r\in\mathbb{N}$ and $1\le\eta\le r$. Then, for
$j=1,\dots,r-\eta+1$, we have
\[
\partial_{j} B_{r,\eta}(x_{1},\ldots,x_{r-\eta+1})=\binom{r}{j}\,B_{r-j,\eta-1}\!\left(x_{1},\ldots,x_{r-j-(\eta-1)+1}\right).
\]
\end{lemma}
\begin{proof} 
If we differentiate
$B_{r,\eta}$ defined in \eqref{eq:def-Bell} with respect to $x_{j}$, the terms with $k_{j}=0$
vanish: $\partial_{j}B_{r,\eta} =\sum_{\mathbf{e}_{j}\le\mathbf{k}\le\boldsymbol{\eta}}r!\frac{k_{j}\,x_{j}^{k_{j}-1}}{k_{j}!(j!)^{k_{j}}}\big(\prod_{\substack{\ell=1\\
\ell\ne j
}
}^{r}\frac{x_{\ell}^{k_{\ell}}}{k_{\ell}!(\ell!)^{k_{\ell}}}\big)\mathbbm{1}_{\{|\mathbf{k}|=\eta,\,\sum_{\ell=1}^{r}\ell k_{\ell}=r\}}(\mathbf{k})$, where  $\mathbf{e}_{j}$ is the unit vector whose $j$-th
component is $1$ and whose other components are $0$. We then extract
the binomial coefficient $\binom{r}{j}=\frac{r!}{j!(r-j)!}$ as follows: 
\begin{align*}
\partial_{j}B_{r,\eta} %
 & =\sum_{\mathbf{e}_{j}\le\mathbf{k}\le\boldsymbol{\eta}}\binom{r}{j}%
 \frac{(r-j)!x_{j}^{k_{j}-1}}{(k_{j}-1)!(j!)^{k_{j}-1}}\bigg(\prod_{\substack{\ell=1\\
\ell\ne j,
}
}^{r}\frac{x_{\ell}^{k_{\ell}}}{k_{\ell}!(\ell!)^{k_{\ell}}}\bigg)\mathbbm{1}_{\{|\mathbf{k}|=\eta,\,\sum_{\ell=1}^{r}\ell k_{\ell}=r\}}(\mathbf{k}).%
\end{align*}
Re-indexing $k_{j}$ to $k_{j}+1$ yields 
\begin{align*}
\partial_{j}B_{r,\eta}&=\binom{r}{j}\sum_{\mathbf{0}\le\mathbf{k}\le\boldsymbol{\eta}-\mathbf{e}_{j}}(r-j)!\frac{x_{j}^{k_{j}}}{k_{j}!(j!)^{k_{j}}}\bigg(\prod_{\substack{\ell=1\\
\ell\ne j
}
}^{r}\frac{x_{\ell}^{k_{\ell}}}{k_{\ell}!(\ell!)^{k_{\ell}}}\bigg)\mathbbm{1}_{\{|\mathbf{k}|=\eta-1,\,\sum_{\ell=1}^{r}\ell k_{\ell}=r-j\}}(\mathbf{k})
\\
 & = \binom{r}{j}\sum_{\mathbf{0}\le\mathbf{k}\le\boldsymbol{\eta}-\mathbf{1}}(r-j)!\frac{x_{j}^{k_{j}}}{k_{j}!(j!)^{k_{j}}}\bigg(\prod_{\substack{\ell=1\\
\ell\ne j
}
}^{r}\frac{x_{\ell}^{k_{\ell}}}{k_{\ell}!(\ell!)^{k_{\ell}}}\bigg)\mathbbm{1}_{\{|\mathbf{k}|=\eta-1,\,\sum_{\ell=1}^{r}\ell k_{\ell}=r-j\}}(\mathbf{k})\\ &=\binom{r}{j} B_{r-j,\eta-1},
\end{align*}
where in the second line we used the fact that given the constraint $|\mathbf{k}|=\eta-1$, the range of summation
may be replaced by $\mathbf{0}\le\mathbf{k}\le\boldsymbol{\eta}-\mathbf{1}$. This completes the proof.
\end{proof}
}
\rev{To derive a higher-order chain rule, we consider the functions $C_{r,\eta}(t)$, defined by evaluating partial Bell polynomials at the derivatives of $\Phi$.} 
Define $C_{0,0}:=B_{0,0}=1$ and $C_{r,-1}:=0$ for all $r\in\mathbb{N}_{0}$.
For $r\in\mathbb{N}$, $C_{r,0}:=B_{r,0}=0$, and for integers $1\le\eta\le r$,
define 
\[
C_{r,\eta}(t):=B_{r,\eta}\bigl(\Phi'(t),\Phi''(t),\ldots,\Phi^{(r-\eta+1)}(t)\bigr).
\]
The function $C_{r,\eta}$ thus defined satisfies the following relation.
\begin{lemma}\label{lem:Bell-recursive-uni} 
For integers \(r\) and \(\eta\) satisfying \(0\leq \eta\leq r\), suppose that
\(\Phi\colon(0,1)\to\mathbb{R}\) 
is \((r-\eta+2)\)-times differentiable. 
Then, we have
\[
C_{r+1,\eta}(t)=\Phi'(t)\,C_{r,\eta-1}(t)+C'_{r,\eta}(t),\quad\text{for all }t\in(0,1).
\]
\end{lemma} 
\begin{proof}
The cases $r=0$ as well as $r\in\mathbb{N}$ and $\eta=0$ are easy
to check. We show the statement for $1\le\eta\leq r$. First, the
partial Bell polynomials satisfy the recurrence relation \rev{\cite[Equation (11.11)]{Charalambides_2002},}
\begin{align*}
  C_{r+1,\eta}(t)&=B_{r+1,\eta}\bigl(\Phi'(t),\ldots,\Phi^{(r-\eta+2)}(t)\bigr)\\
  &=\sum_{\rev{q}=0}^{r-\eta+1}\binom{r}{\rev{q}}\Phi^{(\rev{q}+1)}(t)\;B_{r-\rev{q},\eta-1}\bigl(\Phi'(t),\ldots,\Phi^{(r-\rev{q}-\eta+2)}(t)\bigr).
\end{align*}
\rev{and thus, together with \cref{lem:bell-der},}
differentiating $C_{r,\eta}$ gives 
\begin{align*}
C'_{r,\eta}(t) & =\sum_{j=1}^{r-\eta+1}\partial_{j}B_{r,\eta}\Big|_{\substack{x_{m}=\Phi^{(m)}(t)\\
m=1,\dots,r-\eta+1
}
}\Phi^{(j+1)}(t)\\
 & =\sum_{j=1}^{r-\eta+1}\binom{r}{j}B_{r-j,\eta-1}\bigl(\Phi'(t),\ldots,\Phi^{(r-j-\eta+2)}(t)\bigr)\;\Phi^{(j+1)}(t).
\end{align*}
Hence, subtracting $C_{r+1,\eta}(t)$ from $C'_{r,\eta}(t)$ completes
the proof. 
\end{proof}

For $\boldsymbol{0}\le\boldsymbol{\eta}\le \mathbf{r}$ define 
\[
C_{\mathbf{r},\boldsymbol{\eta}}(\boldsymbol{t}):=\prod_{j=1}^{d}C_{ r_{j},\eta_{j}}(t_{j}).
\]
We will prove the higher-order chain rule using the following analogous formula for smooth functions.
\begin{lemma}\label{lem:faa-di-bruno-smooth}Let $v\in C^{\infty}(\mathbb{R}^{d})$.  
For $\Phi_{j}\in C^{\alpha}((0,1))$, with $\alpha\in \mathbb{N}$, $1\leq j\leq d$,  define $\boldsymbol{\Phi}(\boldsymbol{t})=\big(\Phi_{1}(t_{1}),\dots,\Phi_{d}(t_{d})\big)$
for $\boldsymbol{t}=(t_{1},\dots,t_{d})\in(0,1)^{d}$. For every multi-index
$\mathbf{r}\in\mathbb{N}_{0}^{d}$ with $|\mathbf{r}|_\infty\leq \alpha$
\begin{equation}
D^{\mathbf{r}}(v\circ\boldsymbol{\Phi})(\boldsymbol{t})=\sum_{\boldsymbol{0}\leq\boldsymbol{\eta}\le\mathbf{r}}(D^{\boldsymbol{\eta}}v\circ\boldsymbol{\Phi})(\boldsymbol{t})\;C_{\mathbf{r},\boldsymbol{\eta}}(\boldsymbol{t})\label{eq:faa-di-bruno-smooth}
\end{equation}
holds for every $\boldsymbol{t}\in(0,1)^{d}$, where $\boldsymbol{\eta}\le\mathbf{r}$
means $\eta_{j}\le r_{j}$ for all $j=1,\dots,d$. 
\end{lemma} 
\begin{proof}
We use induction on $|\mathbf{r}|=r_{1}+\dots+r_{d}$. The case $\mathbf{r}=\mathbf{0}$
is trivial. 
Assume that \eqref{eq:faa-di-bruno-smooth} holds for 
all $|\mathbf{r}|\le\beta$ with \(|\mathbf{r}|_{\infty}\leq \alpha\), and let $|\mathbf{r}|=\beta$ with \(|\mathbf{r}|_{\infty}\leq \alpha\).  
If $\beta=d\alpha$, then there is nothing to show. For $\beta<d\alpha$, fix $j\in\{1,\dots,d\}$ such that \(r_j\leq \alpha-1\)
and differentiate \eqref{eq:faa-di-bruno-smooth} with respect to
$t_{j}$: 
\[
D^{\mathbf{r}+\mathbf{e}_{j}}(v\circ\boldsymbol{\Phi})(\boldsymbol{t})=\sum_{\boldsymbol{0}\leq\boldsymbol{\eta}\le\mathbf{r}}\partial_{_{j}}\bigl(v_{\boldsymbol{\eta}}\circ\boldsymbol{\Phi}\bigr)(\boldsymbol{t})C_{\mathbf{r},\boldsymbol{\eta}}(\boldsymbol{t})+\sum_{\boldsymbol{0}\leq\boldsymbol{\eta}\le\mathbf{r}}v_{\boldsymbol{\eta}}(\boldsymbol{\Phi}(\boldsymbol{t}))\,\partial_{_{j}}C_{\mathbf{r},\boldsymbol{\eta}}(\boldsymbol{t}),
\]
where we let $v_{\boldsymbol{\eta}}:=D^{\boldsymbol{\eta}}v$ for
notational ease. 

We now re-index the first sum so that the summation over $\boldsymbol{\eta}$ on the right-hand side runs up to $\mathbf{r}+\mathbf{e}_j$. 
Since only the $j$-th component
of $\boldsymbol{\Phi}$ depends on $t_{j}$, we have 
\(
\partial_{_{j}}\bigl(v_{\boldsymbol{\eta}}\circ\boldsymbol{\Phi}\bigr)(\boldsymbol{t})=v_{\boldsymbol{\eta}+\mathbf{e}_{j}}(\boldsymbol{\Phi}(\boldsymbol{t}))\,\Phi_{j}'(t_{j})\). Moreover, $\partial_{j}C_{\mathbf{r},\boldsymbol{\eta}}=\bigl(\partial_{j}C_{ r_{j},\eta_{j}}\bigr)\prod_{k\ne j}C_{ r_{k},\eta_{k}}=0$
if $\eta_{j}=0$. Combining these, we obtain 
\begin{align*}
&D^{\mathbf{r}+\mathbf{e}_{j}}(v\circ\boldsymbol{\Phi})(\boldsymbol{t})  =\sum_{\substack{\boldsymbol{0}\leq\boldsymbol{\eta}\le\mathbf{r}+\mathbf{e}_{j}\\
\eta_{j}\geq1
}
}v_{\boldsymbol{\eta}}(\boldsymbol{\Phi}(\boldsymbol{t}))\,\Phi_{j}'(t_{j})C_{\mathbf{r},\boldsymbol{\eta}-\mathbf{e}_{j}}(\boldsymbol{t})+\sum_{\boldsymbol{0}\leq\boldsymbol{\eta}\le\mathbf{r}}v_{\boldsymbol{\eta}}(\boldsymbol{\Phi}(\boldsymbol{t}))\,\partial_{_{j}}C_{\mathbf{r},\boldsymbol{\eta}}(\boldsymbol{t})\\
 &\! =\!\!\sum_{\substack{\boldsymbol{0}\leq\boldsymbol{\eta}\le\mathbf{r}+\mathbf{e}_{j}\\
\eta_{j}=r_{j}+1
}
}
    \!\!v_{\boldsymbol{\eta}}(\boldsymbol{\Phi}(\boldsymbol{t}))\,\Phi_{j}'(t_{j})C_{\mathbf{r},\boldsymbol{\eta}-\mathbf{e}_{j}}(\boldsymbol{t})
    +\!
    \sum_{\substack{\boldsymbol{0}\leq\boldsymbol{\eta}\le\mathbf{r}\\
\eta_{j}\geq1
}
}\!v_{\boldsymbol{\eta}}(\boldsymbol{\Phi}(\boldsymbol{t}))\bigl[\Phi_{j}'(t_{j})C_{\mathbf{r},\boldsymbol{\eta}-\mathbf{e}_{j}}(\boldsymbol{t})+\partial_{_{j}}C_{\mathbf{r},\boldsymbol{\eta}}(\boldsymbol{t})\bigr].
\end{align*}
We further rewrite the right-hand side. For the first term, from 
\(B_{r+1,r+1}(x)=x^{r+1}=x\,B_{r,r}(x)\) for  $r\geq0$ and $\eta_{j}-1=r_{j}$,
\begin{align*}
\Phi_{j}'(t_{j})C_{\mathbf{r},\boldsymbol{\eta}-\mathbf{e}_{j}}(\boldsymbol{t})&=\Phi_{j}'(t_{j})C_{ r_{j},r_{j}}(t_{j})\prod_{k\ne j}C_{ r_{k},\eta_{k}}(t_{k})\\&=C_{ r_{j}+1,r_{j}+1}(t_{j})\prod_{k\ne j}C_{ r_{k},\eta_{k}}(t_{k})=C_{\mathbf{r}+\mathbf{e}_{j},\boldsymbol{\eta}}(\boldsymbol{t})
\end{align*}
holds; for the second term, we use \cref{lem:Bell-recursive-uni} to see 
\begin{align*}
\Phi_{j}'(t_{j})C_{\mathbf{r},\boldsymbol{\eta}-\mathbf{e}_{j}}(\boldsymbol{t})+\partial_{_{j}}C_{\mathbf{r},\boldsymbol{\eta}}(\boldsymbol{t}) & =\Bigl(\Phi_{j}'(t_{j})C_{ r_{j},\eta_{j}-1}(t_{j})+\partial_{j}C_{ r_{j},\eta_{j}}(t_{j})\Bigr)\prod_{k\ne j}C_{ r_{k},\eta_{k}}(t_{k})\\
 & =\Bigl(C_{ r_{j}+1,\eta_{j}}(t_{j})\Bigr)\prod_{k\ne j}C_{ r_{k},\eta_{k}}(t_{k}) =C_{\mathbf{r}+\mathbf{e}_{j},\boldsymbol{\eta}}(\boldsymbol{t}).
\end{align*}
Hence, 
\begin{align*}
    D^{\mathbf{r}+\mathbf{e}_{j}}(v\circ\boldsymbol{\Phi})(\boldsymbol{t})&=\sum_{\substack{\boldsymbol{0}\leq\boldsymbol{\eta}\le\mathbf{r}+\mathbf{e}_{j}\\
\eta_{j}=r_{j}+1
}
}v_{\boldsymbol{\eta}}(\boldsymbol{\Phi}(\boldsymbol{t}))\,C_{\mathbf{r}+\mathbf{e}_{j},\boldsymbol{\eta}}(\boldsymbol{t})+\sum_{\substack{\boldsymbol{0}\leq\boldsymbol{\eta}\le\mathbf{r}\\
\eta_{j}\geq1
}
}v_{\boldsymbol{\eta}}(\boldsymbol{\Phi}(\boldsymbol{t}))\,C_{\mathbf{r}+\mathbf{e}_{j},\boldsymbol{\eta}}(\mathbf{t})\\&=\sum_{\boldsymbol{0}\leq\boldsymbol{\eta}\le\mathbf{r}+\mathbf{e}_{j}}v_{\boldsymbol{\eta}}(\boldsymbol{\Phi}(\boldsymbol{t}))\,C_{\mathbf{r}+\mathbf{e}_{j},\boldsymbol{\eta}}(\boldsymbol{t}),
\end{align*}
which is \eqref{eq:faa-di-bruno-smooth} for $\mathbf{r}+\mathbf{e}_{j}$.
This completes the proof. 
\end{proof}

\begin{proposition}\label{prop:faa-di-bruno-comp}For a multi-index
$\mathbf{r}\in\mathbb{N}_{0}^{d}$ with $|\mathbf{r}|_\infty\leq \alpha$, suppose that $v\in L_{\mathrm{loc}}^{1}(\mathbb{R}^{d})$ 
admits weak derivatives $D^{\mathbf{\boldsymbol{\eta}}}v\in L_{\mathrm{loc}}^{1}(\mathbb{R}^{d})$
for $\boldsymbol{\eta}\le\mathbf{r}$, i.e., $\eta_{j}\le r_{j}$
for $1\leq j\leq d$. Let $\Phi_{j}\in C^{\alpha}((0,1))$ be invertible
with $\Phi_{j}^{-1}\in C^{1}((0,1))$ for $1\leq j\leq d$. Define
$\boldsymbol{\Phi}(\boldsymbol{t})=\big(\Phi_{1}(t_{1}),\dots,\Phi_{d}(t_{d})\big)$
for $\boldsymbol{t}=(t_{1},\dots,t_{d})\in(0,1)^{d}$. Then, 
\begin{equation}
D^{\mathbf{r}}(v\circ\boldsymbol{\Phi})(\boldsymbol{t})=\sum_{\boldsymbol{0}\leq\boldsymbol{\eta}\le\mathbf{r}}(D^{\boldsymbol{\eta}}v\circ\boldsymbol{\Phi})(\boldsymbol{t})\;C_{\mathbf{r},\boldsymbol{\eta}}(\boldsymbol{t})\label{eq:faa-di-bruno-Sobolev}
\end{equation}
holds for a.e.~$\boldsymbol{t}\in(0,1)^{d}$. 
\end{proposition} 
\begin{proof}
It suffices to show that the right-hand side of \eqref{eq:faa-di-bruno-Sobolev}
is locally integrable and that
\[
\int_{(0,1)^{d}}(v\circ\boldsymbol{\Phi})(\boldsymbol{t})\,D^{\mathbf{r}}\varphi(\boldsymbol{t})\mathrm{d}\boldsymbol{t}=(-1)^{|\mathbf{r}|}\int_{(0,1)^{d}}\left(\sum_{\boldsymbol{0}\leq\boldsymbol{\eta}\le\mathbf{r}}(D^{\boldsymbol{\eta}}v\circ\boldsymbol{\Phi})(\boldsymbol{t})\;C_{\mathbf{r},\boldsymbol{\eta}}(\boldsymbol{t})\right)\varphi(\boldsymbol{t})\mathrm{d}\boldsymbol{t}
\]
holds for every compactly supported $\varphi$.

Local integrability follows from the boundedness of $(\Phi_{j}^{-1})'$
and the derivatives $\Phi_{j}^{(\tau)}$ on compact subsets of $(0,1)^{d}$,
as well as from the facts that $D^{\boldsymbol{\eta}}v$ is locally
integrable on $\mathbb{R}^{d}$, and that the image of each compact
subset of $(0,1)^{d}$ under $\boldsymbol{\Phi}$ is compact in $\mathbb{R}^{d}$.
To show the equality above, fix $\varphi$ and take an open set $U$ such that $\mathrm{supp}(\varphi)\subset U\Subset(0,1)^{d}$\rev{, i.e., the closure $\overline{U}$ of $U\subset\mathbb{R}^d$ is compact and $\overline{U}\subset (0,1)^{d}$}.
Noting that $(\Phi_{j}^{-1})'$ and $\Phi_{j}^{(\tau)}$ are bounded
on $U$, we invoke \cref{lem:faa-di-bruno-smooth} together
with a standard argument for change-of-variables for Sobolev functions
to conclude that the equality holds on $U$, and thus equivalently
on $(0,1)^{d}$. 
\end{proof}

\rev{Having established these identities for derivatives, we are now in a position to derive bounds for the transformed integrand and its derivatives.} 
Note that, unlike in the common setting for Sobolev spaces, we do
not assume boundedness of the derivatives of $\boldsymbol{\Phi}$,
since the derivatives of the change of variables $\boldsymbol{\Phi}_{\mathrm{cot}}$
we consider are unbounded. 
As a result, and in contrast to standard
results, the integrability of $D^{\mathbf{r}}(v\circ\boldsymbol{\Phi})$
does not follow immediately, and the membership of $(f\cdot\rho)\circ\boldsymbol{\Phi}_{\mathrm{cot}}$ in the Sobolev space is handled separately in the next proposition.
\begin{proposition}\label{prop:g-zero-boudnary}
Define $g(\bst):=f(\boldsymbol{\Psi}_{\mathrm{cot}}(\bst))\rho(\boldsymbol{\Psi}_{\mathrm{cot}}(\bst))\prod_{k=1}^{d}\phi'(t_{k})$ for $f\in H_{\rho}^{\alpha}$ with $\alpha\in\mathbb{N}$.
We have the following: for all multi-indices $\mathbf{r}\in\{0,1,\ldots,\alpha\}^{d}$
\begin{equation}
\|D^{\mathbf{r}}g\|_{L^{2}([0,1]^{d})}\le C\rev{_{d.\alpha}}\|f\|_{H_{\rho}^{\alpha}},
\label{eq:D_g_bound}
\end{equation}
where $C\rev{_{d.\alpha}}>0$ is a constant independent of $\mathbf{r}$ and $f$\rev{ but depends on $d$ and $\alpha$};
for every $\ell=1,\dots,d$, every $\mathbf{r}\in\{0,\ldots,\alpha\}^{d}$
such that $r_{\ell}\leq\alpha-1$, and almost every $(t_{1},\dots,t_{\ell-1},t_{\ell+1},\dots,t_{d})\in(0,1)^{d-1}$,
\begin{equation}
\lim_{t_{\ell}\to0+}D^{\mathbf{r}}g(\bst)=\lim_{t_{\ell}\to1-}D^{\mathbf{r}}g(\bst)=0,\label{eq:g_periodicity}
\end{equation}
and as a consequence, the map $t_{\ell}\mapsto D^{\mathbf{r}}g(\bst)$
extends continuously to $[0,1]$ by zero. Furthermore, for $\mathbf{r}\in\{0,\ldots,\alpha-1\}^{d}$,
\eqref{eq:g_periodicity} holds for every $(t_{1},\dots,t_{\ell-1},t_{\ell+1},\dots,t_{d})\in(0,1)^{d-1}$.
\end{proposition}
\begin{proof}
We first prove \eqref{eq:D_g_bound} for all $\mathbf{r}\in\{0,1,\ldots,\alpha\}^{d}$.
By \cref{prop:faa-di-bruno-comp}, we have 
\begin{equation*}
 D^{\mathbf{s}}\bigl((f\cdot\rho)\circ\boldsymbol{\Psi}_{\mathrm{cot}}\bigr)(\boldsymbol{t})
 =\sum_{\boldsymbol{0}\leq\boldsymbol{\eta}\le\mathbf{s}}\bigl(D^{\boldsymbol{\eta}}(f\cdot\rho)\circ\boldsymbol{\Psi}_{\mathrm{cot}}\bigr)(\boldsymbol{t})\;C_{\mathbf{s},\boldsymbol{\eta}}(\boldsymbol{t}).  
\end{equation*}
for all \(\boldsymbol{0}\leq\mathbf{s}\leq\mathbf{r}\).
Together with the product rule, we get
\begin{align*}
[D^{\mathbf{r}}g](\bst) & =
\sum_{\mathbf{s}\le\mathbf{r}}\binom{\mathbf{r}}{\mathbf{s}}D^{\mathbf{s}}\bigl(
	(f\cdot\rho)\circ
	\boldsymbol{\Psi}_{\mathrm{cot}}
\bigr)(\bst)\,
	\biggl(D^{\mathbf{r}-\mathbf{s}}\prod_{k=1}^{d}\phi'(t_{k})\biggr)\\
 & =\sum_{\mathbf{s}\le\mathbf{r}}\binom{\mathbf{r}}{\mathbf{s}}\bigg(\sum_{\boldsymbol{\eta}\le\mathbf{s}}
 \bigl(D^{\boldsymbol{\eta}}(f\cdot\rho)\circ
 \boldsymbol{\Psi}_{\mathrm{cot}}\bigr)(\bst)\;C_{\mathbf{s},\boldsymbol{\eta}}(\boldsymbol{t})\bigg)\,
 \biggl(D^{\mathbf{r}-\mathbf{s}}\prod_{k=1}^{d}\phi'(t_{k})\biggr),
\end{align*}
the right-hand side of which is a linear combination of 
\[w(\bst):=
\bigg(
    \bigl([D^{\boldsymbol{\lambda}_{1}}f]\circ\boldsymbol{\Psi}_{\mathrm{cot}}\bigr)(\bst)\bigg)\bigg(
    \bigl([D^{\boldsymbol{\lambda}_{2}}\rho]\circ \boldsymbol{\Psi}_{\mathrm{cot}}\bigr)(\bst)\bigg)\prod_{k=1}^{d}
	\prod_{\tau=1}^{\lambda_{3,k}}\big(D^{\tau}\phi(t_{k})\big)^{\ell_{k,\tau}}\]
with some $\boldsymbol{\lambda}_{1},\boldsymbol{\lambda}_{2}\le\mathbf{r}$, $\boldsymbol{\lambda}_{3}=(\lambda_{3,1},\ldots,\lambda_{3,d})\le\mathbf{r}+\boldsymbol{1}$,  
and $\ell_{k,\tau}\in\mathbb{N}_{0}$. 
The $L^{2}([0,1]^{d})$ norm of each
term is then bounded by 
\begin{align*}
&\|w\|_{L^{2}([0,1]^{d})}^{2}  =\int_{\R^{d}}w^{2}(\boldsymbol{\Psi}_{\mathrm{cot}}^{-1}(\bsx))\bigg(\prod_{k=1}^{d}(\phi^{-1})'(x_{k})\bigg)\rd\bsx\\
 & =\int_{\R^{d}}
 	\biggl[
		\Bigl(D^{\boldsymbol{\lambda}_{1}}\!\!\;f(\bsx)\Bigr)\Bigl(D^{\boldsymbol{\lambda}_{2}}\!\!\;\rho(\bsx)\Bigr)
	\biggr]^{2}\bigg(\prod_{k=1}^{d} \bigg(\prod_{\tau=1}^{\lambda_{3,k}}\big([D^{\tau}\phi]\circ\phi^{-1}(x_{k})\big)^{2\ell_{k,\tau}}\bigg)(\phi^{-1})'(x_{k})\bigg)\rd\bsx\\
 & \le
 \|D^{\boldsymbol{\lambda}_{1}}f\|_{L_{\rho}^{2}(\R^{d})}^{2}
 \\
 &\qquad \times
 \sup_{\bsy\in\R^{d}}
 	\Bigl([D^{\boldsymbol{\lambda}_{2}}\rho](\bsy)\Bigr)^{2}(\rho(\bsy))^{-1}
		\biggl(\prod_{k=1}^{d}
		\bigg(\prod_{\tau=1}^{\lambda_{3,k}}\big([D^{\tau}\phi]\circ\phi^{-1}(y_{k})\big)^{2\ell_{k,\tau}}\bigg)(\phi^{-1})'(y_{k})
		\biggr)\\
 & =\|D^{\boldsymbol{\lambda}_{1}}f\|_{L_{\rho}^{2}(\R^{d})}^{2}
 \\
 &\ \ \times
 \prod_{k=1}^{d} \left(\sup_{y_{k}\in\mathbb{R}}\rho(y_{k})\,(\lambda_{2,k}!)\,H_{\lambda_{2,k}}^{2}(y_{k})\bigg(\prod_{\tau=1}^{\lambda_{3,k}}\big([D^{\tau}\phi]\circ\phi^{-1}(y_{k})\big)^{2\ell_{k,\tau}}\bigg)(\phi^{-1})'(y_{k})\right)
<\infty,
\end{align*}
where, in the last line, we used $(\phi^{-1})'(y_{k})=1/\bigl(\pi(1+y_{k}^2)\bigr)$ as well as the fact that $([D^{\tau}\phi]\circ\phi^{-1}(x_{k})\big)^{2\ell_{k,\tau}}$ is a polynomial of degree $2\ell_{k,\tau}(\tau+1)$; see \cref{lem:cot-polynomial}.

We now assume $r_{\ell}\leq\alpha-1$ and prove the vanishing
boundary condition for $w(\bst)$ in the $t_{\ell}$-direction,
from which \eqref{eq:g_periodicity} will follow. 
First, from
\begin{align}
|w(\bst)| & = 
	\bigg|
		\Big(([D^{\boldsymbol{\lambda}_{1}}f]\circ\boldsymbol{\Psi}_{\mathrm{cot}})(\bst)\Big)
		\Big(([D^{\boldsymbol{\lambda}_{2}}\rho]\circ\boldsymbol{\Psi}_{\mathrm{cot}})(\bst)\Big)
		\prod_{k=1}^{d}
\prod_{\tau=1}^{\lambda_{3,k}}\big(D^{\tau}\phi(t_{k})\big)^{\ell_{k,\tau}}\bigg| \nonumber \\ 
 & \le\bigg|
 	\Big(([D^{\boldsymbol{\lambda}_{1}}f]\circ\boldsymbol{\Psi}_{\mathrm{cot}})(\bst)\Big)
	\Big((\rho^{1/2+\varepsilon}\circ\boldsymbol{\Psi}_{\mathrm{cot}})(\bst)\Big)\bigg|   \label{eq:w_two_factors}\\  &\;\;
	\times\bigg|\Big(([D^{\boldsymbol{\lambda}_{2}}\rho]\circ\boldsymbol{\Psi}_{\mathrm{cot}})(\bst)\Big)\Big((\rho^{-1/2-\varepsilon}\circ\boldsymbol{\Psi}_{\mathrm{cot}})(\bst)\Big)\prod_{k=1}^{d}\prod_{\tau=1}^{\lambda_{3,k}}\big(D^{\tau}\phi(t_{k})\big)^{\ell_{k,\tau}}\bigg|\nonumber
\end{align}
 with $\varepsilon\in(0,1/2)$, in view of \cref{prop:1st-order-embedding}, $t_{\ell}\mapsto w(\bst)$ 
admits a continuous representative on $(0,1)$; indeed, 
	following an argument analogous to the latter half of  
the proof of \cref{prop:lattice-affin}, the function
	$v:=[D^{\boldsymbol{\lambda}_{1}}f](\cdot;\boldsymbol{x}_{-\ell})\cdot\rho^{1/2+\varepsilon}(\cdot;\boldsymbol{x}_{-\ell})$
with $(\cdot,\boldsymbol{x}_{-\ell}):=(x_{1},\dots,x_{\ell-1},\cdot,x_{\ell+1},\dots,x_{d})$
is in $W_{\text{mix}}^{1,2}(\mathbb{R})$ for a.e.~$(x_{1},\dots,x_{\ell-1},x_{\ell+1},\dots,x_{d})\in\mathbb{R}^{d-1}$. 
	To see this, for simplicity consider the case $|\mathbf{r}|_{\infty}\leq\alpha-1$.
	Then, $\rho^{1/2+\varepsilon}D^{\boldsymbol{\lambda}_{1}}f\colon\mathbb{R}^{d}\to\mathbb{R}$
	is in $W_{\mathrm{mix}}^{1,2}(\mathbb{R}^{d})$. Indeed, for $\mathbf{k}\in\{0,1\}^{d}$,
	we have
	\begin{align*}
	D^{\mathbf{k}}\bigl(\rho^{1/2+\varepsilon}D^{\boldsymbol{\lambda}_{1}}f\bigr) & =\sum_{\mathbf{j\leq k}}\binom{\mathbf{k}}{\mathbf{j}}D^{\mathbf{j}}(\rho^{1/2+\varepsilon})D^{\mathbf{k}-\mathbf{j}+\boldsymbol{\lambda}_{1}}f\\
	& =\sum_{\mathbf{j\leq k}}\binom{\mathbf{k}}{\mathbf{j}}\Bigl(\prod_{j\in\mathbf{j}}-\bigg(\frac{1}{2}+\varepsilon\bigg)x_{j}\Bigr)\rho^{1/2+\varepsilon}D^{\mathbf{k}-\mathbf{j}+\boldsymbol{\lambda}_{1}}f,
	\end{align*}
	so
	\(
	\bigl\|D^{\mathbf{k}}\bigl(\rho^{1/2+\varepsilon}D^{\boldsymbol{\lambda}_{1}}f\bigr)\bigr\|_{L^{2}(\mathbb{R}^{d})}\leq\sum_{\mathbf{j\leq k}}\binom{\mathbf{k}}{\mathbf{j}}c_{\mathbf{j},\varepsilon}\|\rho^{1/2}D^{\mathbf{k}-\mathbf{j}+\boldsymbol{\lambda}_{1}}f\|_{L^{2}(\mathbb{R}^{d})}<\infty
	\).  
	Showing $v\in W_{\mathrm{mix}}^{1,2}(\mathbb{R})$ \rev{ a.e.~in $\mathbb{R}^{d-1}$} when we only have
	$r_{\ell}\leq\alpha-1$ is analogous. 
	  As a result, noting that $(\Psi_{\mathrm{cot}}^{-1},\dots,\Psi_{\mathrm{cot}}^{-1})\colon\mathbb{R}^{d-1}\to(0,1)^{d-1}$
maps null sets to null sets, we see that $t_{\ell}\mapsto D^{\mathbf{r}}g(\boldsymbol{t})$
admits a continuous representative $(0,1)^{d-1}$-a.e. For $|\mathbf{r}|_{\infty}\leq\alpha-1$, the function 
$\rho^{1/2+\varepsilon}\cdot D^{\boldsymbol{\lambda}_{1}}f$ is in $W_{\text{mix}}^{1,2}(\R^{d})$ and thus the analogous
claim holds everywhere in $(0,1)^{d-1}$.

For \eqref{eq:g_periodicity}, we show that the first factor in \eqref{eq:w_two_factors} involving $f$ is bounded on $\mathbb{R}^{d}$, and that 
the second factor approaches zero as $\bst$ goes to the boundary.
For the first factor, again we use the argument in the proof
of \cref{prop:lattice-affin} 
to deduce $v\in W_{\text{mix}}^{1,2}(\mathbb{R})$, and thus from
\cref{prop:1st-order-embedding} $\sup_{x_{\ell}\in\mathbb{R}}|v(x_{\ell})|<\infty$. 
This holds everywhere in $\mathbb{R}^{d-1}$ if $|\mathbf{r}|_{\infty}\leq\alpha-1$,
and a.e. in $\mathbb{R}^{d-1}$ otherwise.
 For the second factor,
we have 
\begin{align*}
 & \lim_{t_{j}\to0+,1-}\bigg|	
 	\Bigl(([D^{\boldsymbol{\lambda}_{2}}\rho]\circ\boldsymbol{\Psi}_{\mathrm{cot}})(\bst)\Bigr)
	\Bigl((\rho^{-1/2-\varepsilon}\circ\boldsymbol{\Psi}_{\mathrm{cot}})(\bst)
	\Bigr)
		\prod_{k=1}^{d}\prod_{\tau=1}^{\lambda_{3,k}}\big(D^{\tau}\phi(t_{k})\big)^{\ell_{k,\tau}}\bigg|\\
 & =\lim_{|x_{j}|\to\infty}
 	\bigg|
		\Bigl([D^{\boldsymbol{\lambda}_{2}}\rho](\bsx)\Bigr)
		\Bigl(\rho^{-1/2-\varepsilon}(\bsx)\Bigr)
		\prod_{k=1}^{d}\bigg(\prod_{\tau=1}^{\lambda_{3,k}}\big([D^{\tau}\phi]\circ\phi^{-1}(x_{k})\big)^{\ell_{k,\tau}}\bigg)\bigg|\\
 & =\lim_{|x_{j}|\to\infty}
 	\bigg|
		\Bigl(H_{\boldsymbol{\lambda}_{2}}(\bsx)\;\rho^{1/2-\varepsilon}(\bsx)\Bigr)\prod_{k=1}^{d}\bigg(\prod_{\tau=1}^{\lambda_{3,k}}\big([D^{\tau}\phi]\circ\phi^{-1}(x_{k})\big)^{\ell_{k,\tau}}\bigg)\bigg|=0,
\end{align*}
where in the last line, we again used the fact that $([D^{\tau}\phi]\circ\phi^{-1}(x_{k})\big)^{\ell_{k,\tau}}$
 is a polynomial of degree $(\tau+1)\ell_{k,\tau}$; see \cref{lem:cot-polynomial}.
Thus we have proved \eqref{eq:g_periodicity}. 
\end{proof}

	Using \cref{prop:g-zero-boudnary}, we will show that if $f \in H_{\rho}^{\alpha}$, then after a Möbius transformation it belongs to Sobolev-type spaces on the unit cube, where QMC rules yield small errors.   We use two such spaces. 
	The first is the Korobov space, given by
	\begin{align*}
    {H}_{\text{Kor}}^{\alpha}&:=\bigg\{f\in L^2([0,1]^d) \bigg| \|f\|^2_{\text{Kor},\alpha} := \sum_{\mathbf{k}\in\mathbb{Z}^{d}} |\widehat{f}(\mathbf{k})|^2 r_{\alpha}^2(\mathbf{k})<\infty,\\  & \qquad r_{\alpha}(\mathbf{k}):= \prod_{j=1}^d\max(1,|k_j|^\alpha), \;\; \widehat{f}(\mathbf{k}):=\int_{[0,1]^d} f(\boldsymbol{t})\mathrm{e}^{-2\pi\mathrm{i}\boldsymbol{t}\cdot\mathbf{k}}\mathrm{d}\boldsymbol{t}\bigg\}.
    \end{align*}
	The second is the unanchored Sobolev space, defined by \eqref{eq:def-unanchored-alpha-ab} with $\boldsymbol{a}=\boldsymbol{0}$ and $\boldsymbol{b}=\boldsymbol{1}$.

For the Korobov space, the rank-$1$ lattice rule constructed by \cite[Algorithm~3.14]{Dick.J_Kritzer_Pillichshammer_2022_Book_LatticeRules} achieves the convergence rate $N^{-\alpha}(\ln N)^{d\alpha}$. 
   For the unanchored Sobolev space, higher-order digital nets achieve the optimal rate  $N^{-\alpha}(\ln N)^{(d-1)/2}$ \cite{Goda.T_Suzuki_Yoshiki_2018_OptimalOrderQuadrature}.
	This rate is optimal,  including the logarithmic factor,  among all linear and nonlinear quadrature classes that rely solely on function evaluations.  
   QMC rules for $\mathbb{R}^d$  with Möbius transformation inherit these results: those based on rank-$1$ lattice rules achieve the optimal rate up to a logarithmic factor, whereas those based on higher-order digital nets inherit full optimality, including the logarithmic factor, as we show below.

\begin{remark}[Evaluation at the boundary]
	Lattice points 
	and digital nets on the unit cube include 
	boundary points
	such as $\boldsymbol{0}$,
	where the M\"obius transformation is undefined.
	\cref{prop:g-zero-boudnary} tells us that this is not a problem, since the integrand $g(\bst)=f(\boldsymbol{\Psi}_{\mathrm{cot}}(\bst))\rho(\boldsymbol{\Psi}_{\mathrm{cot}}(\bst))\prod_{k=1}^{d}\phi'(t_{k})$ for $f\in H_{\rho}^{\alpha}$ admits an extension by zero to the boundary. 
	Thus,  in the implementation of QMC methods using such  points, boundary values may be set to zero. 
\end{remark}

\begin{proposition}[Lattice
with Möbius transformation] Let $Q_{N}$ be the rank-$1$ lattice
rule with cotangent transform as in \eqref{eq:QMC+Moebius}. Then,
with a generating vector $\boldsymbol{z}$ obtained by the CBC construction
in \cite[Algorithm~3.14]{Dick.J_Kritzer_Pillichshammer_2022_Book_LatticeRules},
$Q_{N}$ satisfies 
\begin{equation}
\sup_{f\neq0}\frac{|I(f)-Q_{N}(f)|}{\|f\|_{H_{\rho}^{\alpha}}}\leq C_{\alpha,d}N^{-\alpha}(\ln N)^{d\alpha}.\label{eq:Moeb-lattice-error}
\end{equation}
\end{proposition} 
\begin{proof}
For $f\in H_{\rho}^{\alpha}$, set $g(\boldsymbol{t})=f\left(\boldsymbol{\Psi}_{\mathrm{cot}}(\boldsymbol{t})\right)\rho\left(\boldsymbol{\Psi}_{\mathrm{cot}}(\boldsymbol{t})\right)\prod_{k=1}^{d}\phi'(t_{k})$
for $\boldsymbol{t}\in(0,1)^{d}$. Fix $\ell\in\{1,\dots,d\}$ and
choose $\mathbf{r}\in\{0,\ldots,\alpha\}^{d}$, $r_{\ell}\leq\alpha-1$.
Noting that $t_{\ell}\mapsto D^{\mathbf{r}}g(\boldsymbol{t})\mathrm{e}^{-2\pi\mathrm{i}\boldsymbol{t}\cdot\mathbf{k}}$
is absolutely continuous on the interval $(0,1)$ for a.e. $\boldsymbol{t}_{-\ell}\in(0,1)^{d}$,
we use integration by parts in $(0,1)$ \cite[Corollary 3.23]{Leoni.G_2017_book_Sobolev_2nd},
as well as uniformly continuity up to $[0,1]$ and the integrability
of functions under consideration on $(0,1)$, to obtain 
\begin{align*}
\int_{0}^{1}D^{\mathbf{r}+\mathbf{e}_{\ell}}g(\boldsymbol{t})\mathrm{e}^{-2\pi\mathrm{i}\boldsymbol{t}\cdot\mathbf{k}}\mathrm{d}t_{\ell} & =2\pi\mathrm{i}k_{\ell}\int_{0}^{1}D^{\mathbf{r}}g(\boldsymbol{t})\mathrm{e}^{-2\pi\mathrm{i}\boldsymbol{t}\cdot\mathbf{k}}\mathrm{d}t_{\ell},
\end{align*}
where we used \cref{prop:g-zero-boudnary} to handle the
boundary terms. Here, $\mathbf{e}_{\ell}$ denotes the $\ell$-th
standard basis in $\mathbb{R}^{d}$. Using \cref{prop:g-zero-boudnary}
once more, we see
\begin{align*}
\|g\|_{\mathrm{Kor},\alpha}^{2} & \leq\sum_{\mathbf{k}\in\mathbb{Z}^{d}}\Bigl(\prod_{j=1}^{d}\sum_{r=0}^{\alpha}(2\pi k_{j})^{2r}\Bigr)\big|\widehat{g}(\mathbf{k})\big|^{2}=\sum_{\mathbf{k}\in\mathbb{Z}^{d}}\sum_{|\mathbf{r}|_{\infty}\leq\alpha}\Bigl(\prod_{j=1}^{d}(2\pi k_{j})^{2r_{j}}\Bigr)\big|\widehat{g}(\mathbf{k})\big|^{2}\\
 & =\sum_{|\mathbf{r}|_{\infty}\leq\alpha}\|D^{\mathbf{r}}g\|_{L^{2}((0,1)^{d})}^{2}\leq C\|f\|_{H_{\rho}^{\alpha}}^{2}<\infty.
\end{align*}
Finally, since $Q_{N}$ satisfies $|I(f)-Q_{N}(f)|\leq C\|g\|_{\mathrm{Kor},\alpha}N^{-\alpha}(\ln N)^{d\alpha}$,
the statement follows. 
\end{proof}
\begin{proposition}[Digital net with Möbius transformation]
Let $Q_{N}$ be the QMC with digital net with cotangent transform
as in \eqref{eq:QMC+Moebius}, where points $(\boldsymbol{t}^{(j)})_{j=1,\dots,N}$
are given by a higher-order net of order $(2\alpha+1)$ with $N=p^{m}$
elements. Then, $Q_{N}$ satisfies 
\begin{equation}
\sup_{f\neq0}\frac{|I(f)-Q_{N}(f)|}{\|f\|_{H_{\rho}^{\alpha}}}\leq C_{\alpha,d}N^{-\alpha}(\ln N)^{(d-1)/2}.\label{eq:Moeb-net-error}
\end{equation}
\end{proposition} 
\begin{proof}
From \cite[Theorem 1]{Goda.T_Suzuki_Yoshiki_2018_OptimalOrderQuadrature},
the QMC under study $Q_{N}$ satisfies 
$|I(f)-Q_{N}(f)|\leq C_{d,\alpha}\|g\|_{\mathscr{H}^{\alpha}([\boldsymbol{0},\boldsymbol{1}])}N^{-\alpha}(\ln N)^{(d-1)/2}$,
where $g(\boldsymbol{t})=f\left(\boldsymbol{\Psi}_{\mathrm{cot}}(\boldsymbol{t})\right)\rho\left(\boldsymbol{\Psi}_{\mathrm{cot}}(\boldsymbol{t})\right)\prod_{k=1}^{d}\phi'(t_{k})$,
and $\|g\|_{\mathscr{H}^{\alpha}([\boldsymbol{0},\boldsymbol{1}])}$
is defined in \eqref{eq:def-unanchored-alpha-ab}. From \cref{prop:g-zero-boudnary},
this norm can be bounded in terms of $\|f\|_{H_{\rho}^{\alpha}}$
as 
\[
\|g\|_{\mathscr{H}^{\alpha}([\boldsymbol{0},\boldsymbol{1}])}^{2}\leq\sum_{\mathbf{r}\in\{0,\dots,\alpha\}^{d}}\|D^{\mathbf{r}}g\|_{L^{2}((0,1)^{d})}^{2}\leq C_{d,\alpha}\|f\|_{H_{\rho}^{\alpha}}^{\rev{2}}.
\]
This completes the proof. 
\end{proof}
\section{Concluding remarks}\label{sec:conclusion}
In this paper, we showed that quasi-Monte Carlo methods with a change of variables achieve optimal worst-case convergence rate in Gaussian Sobolev spaces, whereas sparse-grid quadrature based on the one-dimensional Gauss–Hermite rule is suboptimal in this sense.

These findings motivate further studies in several directions. 
One is to analyse sparse-grid algorithms for functions over $\mathbb{R}^d$ built from univariate rules other than the Gauss--Hermite rule.
A notable alternative is the sparse-grid method using weighted Leja sequences \cite{Ernst.O.G_Sprungk_Tamellini_2021_ExpansionsNodesSparse,Narayan.A_Jakeman_2014_AdaptiveLejaSparse}, which was proposed by Narayan and Jakeman \cite{Narayan.A_Jakeman_2014_AdaptiveLejaSparse}. 
To obtain results analogous to what we showed in this paper, it is worth noting that the fooling function we constructed in \cref{thm:GH-lower-bound-general,thm:lower-isotropic-Smolyak} depend solely on the distribution of quadrature nodes, and not on the quadrature weights.  Hence, understanding the distribution of Leja points for the weight of interest would allow the construction of a corresponding fooling function.

We did not discuss dimension independence of the constant. The implied constant in our error estimate may depend on the dimension $d$, which may be large in practice. 
Addressing this is left for future work. 
A standard approach is to employ the so-called weighted spaces, as  introduced by Hickernell \cite{Hickernell.F.J_1996_RKHS}, under which QMC methods can attain dimension-independent convergence rates with constants that are likewise independent of dimension, as shown by Sloan and Woźniakowski \cite{Sloan.I.H_Wozniakowski_1998_iff}. 
 These ideas have been widely applied in the analysis of parametric partial differential equations \cite{Graham.I_etal_2015_Numerische, Herrmann.L_Schwab_2019_local_lognormal, KaarniojaEtAl.V_2022_FastApproximationPeriodic, Kazashi.Y_2019_product, Robbe.P_etal_2017_MultiIndexQuasiMonteCarlo}.
In fact, the results in \cite{Sloan.I.H_Wozniakowski_1998_iff} show that such weights are even necessary in the setting considered there, suggesting that additional structure is required to make the constants in our bounds dimension-independent.

This paper focused on Sobolev functions. Another important regime in high dimension is that of analytic functions \cite{Chkifa.A_etal_2014_HighDimensionalAdaptiveSparse,
Irrgeher.C_etal_2015_HermiteAnalytic,
Dung.D_etal_2023_AnalyticitySparsityUncertainty}, 
	where the question of optimal rates in this setting remains largely open. 
	 In establishing the lower bound for the sparse-grid Gauss–Hermite rule in this paper, a crucial component was the fooling function constructed for univariate Sobolev functions \cite{Kazashi.Y_Suzuki_Goda_2023_SuboptimalityGaussHermite}. 
	A similar technique for analytic functions is available \cite{Goda.T_Kazashi_Tanaka_2024_HowSharpAreError,MR1427714}, which may prove useful for further analysis in this regime.

\emergencystretch=1em
\printbibliography[heading=siamrefheading]

\end{document}